\documentclass[bj]{imsart}

\usepackage[T1]{fontenc}
\usepackage[utf8]{inputenc}
\usepackage{lmodern}
\usepackage[bookmarks={true},bookmarksopen={true}]{hyperref}
\usepackage{array}
\usepackage{amsmath}
\usepackage{amssymb}
\usepackage{amsthm}
\let\ltxcup\cup
\let\ltxcap\cap
\let\ltxbigcup\bigcup
\let\ltxbigcap\bigcap
\usepackage{mathabx}
\usepackage{enumitem}
\usepackage{mathrsfs}
\usepackage{pgf}
\usepackage{bbm}
\usepackage{fullpage}

\numberwithin{equation}{section}

\renewcommand{\proofname}{Proof}

\renewcommand{\leq}{\leqslant}
\renewcommand{\geq}{\geqslant}
\renewcommand{\cup}{\ltxcup}
\renewcommand{\cap}{\ltxcap}
\renewcommand{\bigcup}{\ltxbigcup}
\renewcommand{\bigcap}{\ltxbigcap}
\renewcommand{\tilde}{\widetilde}
\renewcommand{\epsilon}{\varepsilon}

\newcommand{\area}{s}

\newcommand{\A}{\mathscr{A}}
\newcommand{\B}{\mathscr{B}}

\newcommand{\der}{\mathrm{d}}
\newcommand{\E}{\mathbb{E}}
\newcommand{\F}{\mathcal{E}}

\newcommand{\N}{\mathbb{N}}
\newcommand{\Z}{\mathbb{Z}}

\newcommand{\R}{\mathbb{R}}

\newcommand{\cara}[1]{\mathbbm{1}_{#1}}
\newcommand{\vol}[1]{|#1|}

\newcommand{\card}{\mbox{card}}

\newcommand{\diam}{\mbox{diam}}

\newcommand{\dist}{\mbox{dist}}

\newcommand{\tr}{\mbox{Tr}}

\newcommand{\Sp}{\mbox{Sp}}
\newcommand{\adj}{\mbox{adj}}
\newcommand{\var}{\mbox{Var}}
\newcommand{\cov}{\mbox{Cov}}

\newcommand{\cvlaw}{\overset{\mathcal{L}}{\longrightarrow}}
\newcommand{\defeq}{:=}
\newcommand{\sizeR}{\right}
\newcommand{\sizeL}{\left}

\newcommand{\lun}[1]{|#1|_1}

\theoremstyle{plain}
\newtheorem{prop}{Proposition}[section]
\newtheorem{lem}[prop]{Lemma}
\newtheorem{defin}[prop]{Definition}
\newtheorem{defin/prop}[prop]{Definition/Proposition}
\newtheorem{theo}[prop]{Theorem}
\newtheorem{corr}[prop]{Corollary}

\begin{document}

\begin{frontmatter}
\title{Mixing properties and central limit theorem for associated point processes}
\runtitle{Mixing properties and central limit theorem for associated point processes}

\begin{aug}
  \author{\fnms{Arnaud}  \snm{Poinas}\thanksref{a,e1}\ead[label=e1,mark]{arnaud.poinas@univ-rennes1.fr}}
  \author{\fnms{Bernard} \snm{Delyon}\thanksref{a,e3}\ead[label=e3,mark]{bernard.delyon@univ-rennes1.fr}}
  \author{\fnms{Fr\'ed\'eric} \snm{Lavancier}\thanksref{b,e2}\ead[label=e2,mark]{Frederic.Lavancier@univ-nantes.fr}}
  \runauthor{A. Poinas and B. Delyon and F. Lavancier}
  \affiliation{Universit\'e de Rennes}
	\address[a]{IRMAR -- Campus de Beaulieu - Bat. 22/23 -- 263 avenue du G\'en\'eral Leclerc --  35042 Rennes -- France}
	\printead{e1}\\
	\printead{e3}
  \address[b]{Laboratoire de Math\'ematiques Jean Leray -- BP 92208 -- 2, Rue de la Houssini\`ere -- F-44322 Nantes Cedex 03 -- France. Inria, Centre Rennes  Bretagne Atlantique, France.}
\printead{e2}
  \end{aug}

\begin{abstract} 
Positively (resp. negatively) associated point processes are a class of point processes that induce attraction (resp. inhibition) between the points. As an important example, determinantal point processes (DPPs) are negatively associated. We prove $\alpha$-mixing properties for associated spatial point processes by controlling their $\alpha$-coefficients in terms of the first two intensity functions. A central limit theorem for functionals of associated point processes is deduced, using both the association and the $\alpha$-mixing properties. We discuss in detail the case of DPPs, for which we obtain the limiting distribution of sums, over subsets of close enough points of the process, of any bounded function of the DPP. As an application, we get the asymptotic properties of the parametric two-step estimator of some inhomogeneous DPPs. 
\end{abstract}


\end{frontmatter}
\section{Introduction}

Positive association (PA)  and negative association (NA) \cite{NA_Int,PA_Int} are properties that quantify the  dependence between random variables.  They have found many applications  in  limit theorems for random fields \cite{BulinskiRate,CLTRFNA}. 
Even if the extension of PA to point processes have been used in analysis of functionals of random measures \cite{AppNa1,AppNa2}, there are no general applications of PA or NA to limit theorems for point processes. We contribute in this paper to this aspect for spatial point processes on $\R^d$.
We especially discuss in detail the  case of determinantal point processes (DPPs for short), that are an important example of negatively associated point processes. DPPs are a type of repulsive point processes that were first introduced by Macchi~\cite{Macchi} in 1975 to model systems of fermions in the context of quantum mechanics. They have been extensively studied in Probability theory with applications ranging from random matrix theory to non-intersecting random walks, random spanning trees and more (see~\cite{Hough}). From a statistical perspective, DPPs have applications in machine learning~\cite{MachLearn}, telecommunication~\cite{Telecom1,Telecom2,gomez_case_2015}, biology, forestry~\cite{Lavancier} and computational statistics~\cite{BaHa16Sub}.

As a first result, we relate the association property of a point process to its $\alpha$-mixing properties.
First introduced in \cite{AlphaMix}, $\alpha$-mixing is a measure of dependence between random variables, which is actually  more popular than PA or NA. It has been used extensively to prove central limit theorems for dependent random variables~\cite{Bolt, Doukhan, Guyon, Ibragimov, AlphaMix}. More details about mixing can be found in~\cite{Survey, Doukhan}.
We derive in Section~\ref{sec:NA} an important covariance inequality for associated point processes (Theorem~\ref{NAmain}), that turns out to be very similar to inequalities established in \cite{WeakDep} for weakly dependent continuous random processes. We show that this inequality implies $\alpha$-mixing and precisely allows to control the $\alpha$-mixing coefficients by the first two intensity functions of the point process. This result for point processes is in contrast with the case of random fields where it is known that association does not imply $\alpha$-mixing in general (see Examples~5.10-5.11 in~\cite{BulinskiRate}). However, this implication holds true for integer-valued random fields (see \cite{DiscreteDep} or \cite{BulinskiRate}). As explained in \cite{DiscreteDep}, this is because the $\sigma$-algebras generated by countable sets are much poorer than $\sigma$-algebras generated by continuous sets. In fact, by this aspect and some others (for instance our proofs boil down to the control of the number of points in bounded sets), point processes are very similar to discrete processes.

We then establish in Section~\ref{sec:CLT} a general central limit theorem (CLT) for random fields defined as a function of an associated point process (Theorem~\ref{CLT}). A standard method for proving  this kind of theorem is to rely on sufficiently fast decaying $\alpha$-mixing coefficients along with some moment assumptions. We use an alternative procedure that exploits both the mixing properties and the association property. This results in weaker assumptions on the underlying point process, that can have slower decaying mixing coefficients.  This improvement allows in particular to include all standard DPPs, some of them being otherwise excluded with the first approach (like for instance DPPs associated to the Bessel-type kernels~\cite{Biscio}).

Section~\ref{sec:DPP} discusses in detail the case of DPPs, where we derive a tight explicit bound for their $\alpha$-mixing coefficients and prove a central limit theorem for certain functionals of a DPP (Theorem~\ref{CLTDPP}). Specifically, these functionals write as a sum of a bounded function of the DPP, over subsets of close-enough points of the DPP. A particular case concerns sums over $p$-tuple of close enough points of the DPP, which are frequently encountered in asymptotic inference. Limit theorems in this setting have been established in \cite{Sosh} when $p=1$, and in \cite{BiscioMix} for stationary DPPs and $p\geq 1$.  We thus extend these studies to sums over any subsets and without the stationary assumption. As a statistical application, we consider the parametric estimation of second-order intensity reweighted stationary DPPs. These DPPs have an inhomogeneous first order intensity, but translation-invariant higher order (reweighted) intensities. We prove that the two-step estimator introduced in~\cite{waag}, designed for this kind of inhomogeneous point process models, is consistent and asymptotically normal when applied to  DPPs. 

\section{Associated point processes and \texorpdfstring{$\alpha$}{alpha}-mixing} \label{sec:NA}
\subsection{Notation}
In this paper, we consider locally finite simple point processes on $\R^d$, for a fixed $d\in\N$.
Some theoretical background on point processes can be found in~\cite{DV,moeller:waagepetersen:04}. 
 We denote by $\Omega$ the set of locally finite point configurations in $\R^d$. For $X\in\Omega$ and $A\subset\R^d$, we write
 $$N(A)\defeq\card(X\cap A)$$
 for the random variable representing the number of points of $X$ that fall in $A$. We also denote by $\B(A)$ the Borel $\sigma$-algebra of $A$ and by $\F(A)$ the $\sigma$-algebra generated by $X\cap A$, defined by
$$\F(A)\defeq\sigma(\{X\in\Omega: N(B)=m\},B\in\B(A),m\in\N).$$ 
 
 The notation $|.|$ will have a different meaning depending on the object it is applied. For $x\in\R^d$,  $|x|$ stands for the euclidean norm. For a set $A\subset\R^d$, $|A|:=\int_A \der x$ is the euclidean volume of $A$, and for a set $I\subset\Z^d$ we write $|I|$ for the cardinal of $I$. For $A,B$ two subsets of $\R^d$ (resp. $\Z^d$) we define $\dist(A,B)$ as $\inf_{x\in A,y\in B}|y-x|$ and $\diam(A)$ as $\sup_{x,y\in A}|y-x|$ where $|.|$ is the associated norm on $\R^d$ (resp. $\Z^d$). For $i\in\Z^d$, $\lun{i}$ denotes the $\ell_1$-norm. Finally, we write $\mathcal{B}(x,r)$ for the  euclidean  ball centred at $x$ with radius $r$ and $\|.\|_p$ for the $p$-norm of random variables and functions where $1 \leq p\leq \infty$.

We recall that the intensity functions of a point process (when they exist), with respect to the Lebesgue measure,  are defined as follows.
\begin{defin}
Let $X\in\Omega$ and $n\geq 1$ be an integer. If there exists a non-negative function $\rho_n:(\R^d)^n\rightarrow\R$ such that
$$\E\left [ \sum^{\neq}_{x_1,\cdots,x_n\in X} f(x_1,\cdots,x_n) \right ]=\int_{(\R^d)^n}f(x_1,\cdots,x_n)\rho_n(x_1,\cdots,x_n)\der x_1\cdots\der x_n$$
for all locally integrable functions $f:(\R^d)^n\rightarrow\R$ then $\rho_n$ is called the $n$th order intensity function of $X$.
\end{defin}
In particular, $\rho_n(x_1,\cdots,x_n)\der x_1\cdots\der x_n$ can be viewed as the probability that $X$ has a point in each of the infinitesimally small sets around $x_1,\cdots,x_n$ with volumes $\der x_1,\cdots,\der x_n$ respectively. 

We further introduce the notation 
\begin{equation}\label{defD}
D(x,y)\defeq \rho_2(x,y)-\rho_1(x)\rho_1(y).
\end{equation}
This quantity is involved in the following equality, deduced from the previous definition and used several times throughout the paper:
\begin{equation}\label{eq:cov}
\cov(N(A),N(B))=\int_{A\times B}D(x,y)\der x\der y.
\end{equation}

\subsection{Negative and positive association}

Our goal in this section is to prove a crucial covariance inequality and to deduce an $\alpha$-mixing property for associated point processes. We recall that associated point processes are defined the following way (see Definitions 2.11-2.12 in \cite{DefAssoc} for example).

\begin{defin}
A point process $X$ is said to be negatively associated (NA for short) if, for all families of pairwise disjoint Borel sets $(A_i)_{1\leq i\leq k}$ and $(B_i)_{1\leq i\leq l}$ such that
\begin{equation} \label{condNA} 
(\cup_i A_i) \cap (\cup_j B_j) = \emptyset
\end{equation}
and for all coordinate-wise increasing functions $F:\N^k\mapsto\R$ and $G:\N^l\mapsto\R$ it satisfies
\begin{multline} \label{NAvec}
\E[F(N(A_1),\cdots,N(A_k))G(N(B_1),\cdots,N(B_l))] \\
\leq\E[F(N(A_1),\cdots,N(A_k))]\E[G(N(B_1),\cdots,N(B_l))].
\end{multline}
Similarly, a point process is said to be positively associated (PA for short) if it satisfies the reverse inequality for all families of pairwise disjoint Borel sets $(A_i)_{1\leq i\leq k}$ and $(B_i)_{1\leq i\leq l}$ (but not necessarily satisfying~(\ref{condNA})).\\
If a point process is NA or PA it is said to be associated.
\end{defin} 

The main difference between the definition of PA and NA is the restriction~(\ref{condNA}) that only affects NA point processes. Notice that without~(\ref{condNA}), $\E[N(A)]^2\leq\E[N(A)^2]$ contradicts~(\ref{NAvec}) hence the need to consider functions depending on disjoint sets for NA point processes.

These inequalities extend to the more general case of functionals of point processes. The first thing we need is a more general notion of increasing functions. We associate to $\Omega$ the partial order $X\leq Y$ iff $X\subset Y$ and then call a function on $\Omega$ increasing if it is increasing respective to this partial order. 
The association property can then be extended to these functions. A proof in a general setting can be found in~\cite[Lemma 3.6]{Lyons}
but we give an alternative elementary one in Appendix~\ref{proofAss}.

\begin{theo} \label{trueNA}
Let $X$ be a NA point process on $\R^d$ and $A,B$ be disjoint subsets of $\R^d$. Let $F:\Omega\mapsto\R$ and $G:\Omega\mapsto\R$ be bounded increasing functions, then
\begin{equation} \label{CNA}
\E[F(X\cap A)G(X\cap B)]\leq\E[F(X\cap A)]\E[G(X\cap B)].
\end{equation}
If $X$ is PA then, for all $A,B\subset \R^d$ not necessarily disjoint,
\begin{equation} \label{CPA}
\E[F(X\cap A)G(X\cap B)]\geq\E[F(X\cap A)]\E[G(X\cap B)].
\end{equation}
\end{theo}

Association is a very strong dependence condition. As proved in the following theorem, it implies a strong covariance inequality that is only controlled by the behaviour of the first two intensity functions of $X$ (assuming their existence). To state this result, we need to introduce the following seminorm for functionals over point processes.

\begin{defin}
For any $A\subset\R^d$,  $\|.\|_A$ is the seminorm on the functions $f:\Omega\mapsto\mathbb{C}$ defined by
$$\|f\|_A\defeq\sup_{\substack{X\in\Omega, X\subset A \\ x\in A}}|f(X)-f(X\cup\{x\})|.$$
\end{defin}

Note that $\|.\|_A$ is a Lipschitz norm in the sense that it controls the way $f(X)$ changes when a point is added to $X\cap A$.

\begin{theo} \label{NAmain}
Let $X$ be an associated point process and $A,B\subset\R^d$ be two disjoint bounded subsets. Let $f:\Omega\rightarrow\R$ and $g:\Omega\rightarrow\R$ be two functions such that $f(X\cap A)$ and $g(X\cap B)$ are bounded, then
\begin{equation} \label{NAbound}
|\cov(f(X\cap A),g(X\cap B))|\leq \|f\|_A\|g\|_B|\cov(N(A),N(B))|.
\end{equation}
Moreover, if $X$ is PA then it also satisfies the same inequality for all $A,B\subset\R^d$ not necessarily disjoint.
\end{theo}

\begin{proof} The proof mimics the one from~\cite{Bulinski} for associated random fields. We only consider the case of NA point processes but the PA case can be treated in the same way.\\
Consider the functions $f_+,f_-:\Omega\rightarrow\R$, $\F(A)$-measurable, and $g_+,g_-:\Omega\rightarrow\R$, $\F(B)$-measurable, defined by
$$\left \{ \begin{tabular}{l} $f_{\pm}(X)=f(X\cap A)\pm \|f\|_A N(A),$ \\  $g_{\pm}(X)=g(X\cap B)\pm \|g\|_B N(B).$ \end{tabular} \right .$$
For all $x\in A\backslash X$, $f_+(X\cup\{x\})-f_+(X)=f(X\cup\{x\}\cap A)-f(X\cap A)+\|f\|_A$ which is positive by definition of $\|f\|_A$. $f_+$ is thus an increasing function. With the same reasoning, $g_+$ is also increasing and $f_-,g_-$ are decreasing. $f_+$ is not bounded but it is non-negative and almost surely finite so it can be seen as an increasing limit of the sequence of functions $\min(f_+,k)$ when $k$ goes to infinity. These functions are non-negative, increasing and bounded so for any $k$ and any bounded increasing function $g$,~(\ref{CNA}) applies where $f$ is replaced by $\min(f_+,k)$. By a limiting argument, the same inequality holds true for $f=f_+$. We can also treat the other functions the same way and we get from~(\ref{CNA})
$$\cov(f_+(X),g_+(X))\leq 0~~\mbox{and}~~\cov(f_-(X),g_-(X))\leq 0.$$
Since these expressions are equal to
\begin{multline*}
\cov(f_\pm(X),g_\pm(X))=\cov(f(X\cap A),g(X\cap B))+\|f\|_A\|g\|_B\cov(N(A),N(B))\\
\pm (\|g\|_B\cov(f(X\cap A),N(B))+\|f\|_A\cov(N(A),g(X\cap B))),
\end{multline*}
adding these two expressions together yields the upper bound in~(\ref{NAbound}):
$$\cov(f(X\cap A),g(X\cap B))\leq -\|f\|_A\|g\|_B\cov(N(A),N(B)).$$
The lower bound is obtained by replacing $f$ by $-f$ in the previous expression.
\end{proof}

A similar inequality as in Theorem~\ref{NAmain} can also be obtained for complex-valued functions since $\|\Re(f)\|_A$ and $\|\Im(f)\|_A$ are bounded by $\|f\|_A$, where $\Re(f)$ and $\Im(f)$ refer to the real and imaginary part of $f$ respectively.

\begin{corr} \label{mainC}
Let $X$ be an associated point process and $A,B\subset\R^d$ be two disjoint bounded subsets. Let $f:\Omega\rightarrow\mathbb{C}$ and $g:\Omega\rightarrow\mathbb{C}$ be two functions such that $f(X\cap A)$ and $g(X\cap B)$ are bounded, then
$$|\cov(f(X\cap A),g(X\cap B))|\leq 4\|f\|_A\|g\|_B|\cov(N(A),N(B))|.$$
Moreover, if $X$ is PA then it also satisfies the same inequality for all $A,B\subset\R^d$ not necessarily disjoint.
\end{corr}

If  the first two intensity functions of $X$ are well-defined then $D$ in \eqref{defD} is well-defined.
As a consequence of Theorem~\ref{NAmain} and from \eqref{eq:cov}, if $|D(x,y)|$ vanishes fast enough when $|y-x|$ goes to infinity then any two events respectively in $\F(A)$ and $\F(B)$ will get closer to independence as $\dist(A,B)$ tends to infinity, as specified by the following corollary.

\begin{corr} \label{mainDPP}
Let $X$ be an associated point process on $\R^d$ whose first two intensity functions are well-defined. Let $A,B$ be two bounded disjoint sets of $\R^d$ such that $\dist(A,B)\geq r$. Then, for all functions $f:\Omega\rightarrow\mathbb{R}$ and $g:\Omega\rightarrow\mathbb{R}$ such that $f(X\cap A)$ and $g(X\cap B)$ are bounded,
\begin{equation} \label{alpha+}
|\cov(f(X\cap A),g(X\cap B))|\leq \area_d\vol{A}\,\|f\|_A\|g\|_B \int_r^{\infty} t^{d-1}\sup_{|x-y|=t}|D(x,y)|\der t,
\end{equation}
where $\area_d$ is the $(d-1)$-dimensional area measure of the unit sphere in $\R^d$. Moreover, if $f$ and/or $g$ are complex-valued functions, the same inequality holds true with an extra factor 4 on the right hand side.
\end{corr}
\begin{proof}
Consider $A,B$ to be two bounded disjoint sets of $\R^d$ such that $\dist(A,B)\geq r$ then, from \eqref{eq:cov},
\begin{align*}
|\cov(N(A),N(B))|& =\left |\int_{A\times B} D(x,y) \der x\der y\right |\\
&\leq |A|\sup_{x\in A}\int_B |D(x,y)|\der y\\
&\leq |A|\sup_{x\in A}\int_{\mathcal{B}(x,r)^c} |D(x,y)|\der y\\
&\leq |A|\sup_{x\in A}\int_{\mathcal{B}(x,r)^c} \sup_{\substack{u\in \R^d \\ |u-x|=|y-x|}}|D(x,u)|\der y\\
&\leq |A|\area_d\int_r^{\infty} t^{d-1}\sup_{|u-v|=t}|D(u,v)|\der t.
\end{align*}
The final result is then a consequence of Theorem~\ref{NAmain} and Corollary~\ref{mainC}.
\end{proof}

\subsection{Application to \texorpdfstring{$\alpha$}{alpha}-mixing}
Let us first recall some generalities about mixing. Consider a probability space $(\mathcal{X},\mathcal{F},\mathbb{P})$ and $\A,\B$ two sub $\sigma$-algebras of $\mathcal{F}$. 
The $\alpha$-mixing coefficient is defined as the following measure of dependence between $\A$ and $\B$:
$$\alpha(\A,\B)\defeq\sup\{|\mathbb{P}(A\cap B)-\mathbb{P}(A)\mathbb{P}(B)|: A\in\A, B\in\B\}.$$
In particular, $\A$ and $\B$ are independent iff $\alpha(\A,\B)=0$. This definition leads to the essential covariance inequality due to Davydov \cite{Davydov} and later generalised by Rio \cite{Rio}: For all random variables $X,Y$ measurable with respect to $\A$ and $\B$ respectively,
\begin{equation} \label{covalpha}
|\cov(X,Y)|\leq 8\alpha^{1/r}(\A,\B)\|X\|_p\|Y\|_q,~~\mbox{where}~p,q,r\in [1,\infty]~\mbox{and}~\frac{1}{p}+\frac{1}{q}+\frac{1}{r}=1.
\end{equation}
This definition is adapted to random fields the following way (see~\cite{Doukhan} or~\cite{Guyon}). Let $Y=(Y_i)_{i\in\Z^d}$ be a random fields on $\Z^d$ and define
$$\alpha_{p,q}(r)\defeq\sup\{\alpha(\sigma(\{Y_i,i\in A\}),\sigma(\{Y_i,i\in B\})) : \vol{A}\leq p, \vol{B}\leq q, \mbox{dist}(A,B)>r\}$$
with the convention $\alpha_{p,\infty}(r)=\sup_q\alpha_{p,q}(r)$. The coefficients $\alpha_{p,q}(r)$ describe how close two events happening far enough from each other are from being independent. The parameters $p$ and $q$ play an important role since, in general, we cannot get any information directly on the behaviour of $\alpha_{\infty,\infty}(r)$.

We can adapt this definition to point processes the following way. For a point process $X$ on $\R^d$, define
$$\alpha_{p,q}(r)\defeq\sup\{\alpha(\F(A),\F(B)) : \vol{A}\leq p, \vol{B}\leq q, \mbox{dist}(A,B)>r\}$$
with the convention $\alpha_{p,\infty}(r)=\sup_q\alpha_{p,q}(r)$.

As a consequence of Corollary~\ref{mainDPP},
the $\alpha$-mixing coefficients of an associated point process tend to $0$ when $D(x,y)$ vanishes fast enough as $|y-x|$ goes to infinity. More precisely, we have the following inequalities.

\begin{prop} \label{sidePP}
Let $X$ be an associated point process on $\R^d$ whose first two intensity functions are well-defined, then for all $p,q>0$,
\begin{equation} \label{alpha}
\left \{ \begin{tabular}{l} $\displaystyle\alpha_{p,q}(r)\leq pq\sup_{|x-y|\geq r}|D(x,y)|$, \\  $\displaystyle\alpha_{p,\infty}(r)\leq p\area_d\int_r^{\infty} t^{d-1}\sup_{|x-y|=t}|D(x,y)|\der t$. \end{tabular} \right .
\end{equation}
\end{prop}

\begin{proof}
We can write
$$\alpha(\F(A),\F(B))=\sup_{\substack{\A\in\F(A) \\ \B\in\F(B)}}\cov(\cara{\A}(X\cap A),\cara{\B}(X\cap B))$$
so Proposition \ref{sidePP} is a direct consequence of Theorem~\ref{NAmain} and Corollary~\ref{mainDPP} applied to indicator functions.
\end{proof}

\section{Central limit theorem for associated point processes} \label{sec:CLT}

Consider the lattice $(x_i)_{i\in\Z^d}$ defined by $x_i=R\cdot i$, where $R>0$ is a fixed constant. 
We denote by $C_i$, $i\in\Z^d$, the $d$-dimensional cube with centre $x_i$ and side length $s$, where $s\geq R$ is another fixed constant. Note that the union of these cubes forms a covering of $\R^d$. 
Let $X$ be an associated point process and $(f_i)_{i\in\Z^d}$ be a family of real-valued measurable functions defined on $\Omega$. We consider the centred random field $(Y_i)_{i\in\Z^d}$ defined by 
\begin{equation} \label{defYi}
Y_i\defeq f_i(X\cap C_i)-\E[f_i(X\cap C_i)],~~i\in\Z^d,
\end{equation}
and we are interested in this section by the asymptotic behavior of $S_n\defeq\sum_{i\in I_n}Y_i$, where $(I_n)_{n\in\N}$ is a sequence of strictly increasing finite domains of $\Z^d$.

As a consequence of Proposition~\ref{sidePP}, we could directly use one of the different CLT for $\alpha$-mixing random fields that already exist in the literature \cite{Bolt, Doukhan, Guyon} to get the asymptotic distribution of $S_n$. But, the coefficients $\alpha_{p,\infty}$ decreasing much slower than the coefficients $\alpha_{p,q}$, this would imply an unnecessary strong assumption on $D$. Precisely, this would require $D(x,y)$ to decay at a rate at least $o(|y-x|^{-2(d+\epsilon) \frac{2+\delta}{\delta}})$, where $\epsilon>0$ and $\delta$ is a positive constant depending on the behaviour of the moments of $X$.
In the next theorem, we bypass this issue by exploiting both the behaviour of the mixing coefficients $\alpha_{p,q}$ when $p<\infty$ and $q<\infty$, and the association property through inequality~(\ref{alpha+}). We show that we can still get a CLT when $D(x,y)$ decays at a rate $o(|y-x|^{-(d+\epsilon) \frac{2+\delta}{\delta}})$. This improvement is important to include DPPs with a slow decaying kernel, thus inducing more repulsiveness, such as Bessel-type kernels, see the applications to DPPs in Section~\ref{sec:CLTDPP} and especially the discussion at the end of the section. Let us also remark that another technique, based on the convergence of moments, is sometimes used to establish a CLT for point processes. This has been exploited especially for Brillinger mixing point processes in \cite{jolivetTCL, StellaKleinempirical2011} and other papers. As an example, DPPs have been proved to be Brillinger mixing in \cite{BiscioMix, Zweirich}. However, this condition applies to stationary point processes only.

\begin{theo} \label{CLT}
Consider the random field $Y$ given by \eqref{defYi}, a sequence $(I_n)_{n\in\N}$ of strictly increasing finite domains of $\Z^d$ and $S_n = \sum_{i\in I_n}Y_i$. Let $\sigma^2_n\defeq\var(S_n)$. Assume that for some $\epsilon,\delta>0$ the following conditions are satisfied:
\begin{enumerate}[label=(C\arabic*)]
\item $X$ is an associated point process on $\R^d$ whose first two intensity functions are well-defined; \label{hyp0}
\item $\sup_{i\in\Z^d} \|Y_i\|_{2+\delta}=M< \infty$; \label{hyp1}
\item $\sup_{|x-y|\geq r}|D(x,y)|=\underset{r\rightarrow \infty}{o}(r^{-(d+\epsilon) \frac{2+\delta}{\delta}})$ where $D$ is given by \eqref{defD}; \label{hyp2}
\item $\liminf_{n}|I_n|^{-1}\sigma_n^2>0$. \label{hyp3}
\end{enumerate}
Then
$$\frac{1}{\sigma_n}S_n\cvlaw\mathcal{N}(0,1).$$
\end{theo}

\begin{proof}
First, we notice that $Y$ inherits its strong mixing coefficients from $X$. This is due to the fact that we have $\sigma(\{Y_i: i\in I\})\subset\F(\bigcup_{i\in I}C_i)$ for all $I\subset\Z^d$ as a consequence of~(\ref{defYi}). Moreover, we have $\dist(C_i,C_j)\geq\frac{1}{\sqrt{d}}(\lun{i-j}R-sd)$ as a consequence of Lemma~\ref{distcarre}, and since $|\bigcup_{i\in I}C_i|\leq s^d|I|$, this gives us the inequality
$$\forall p,q>0,~\forall r>\frac{sd}{R},~~\alpha^Y_{p,q}(r)\leq \alpha^X_{ps^d,qs^d}\left(\frac{1}{\sqrt{d}}(rR-sd)\right),$$
where we denote by $\alpha^X,\alpha^Y$ the $\alpha$-mixing coefficients of $X$ and $Y$ respectively. In particular, conditions~\ref{hyp0}, \ref{hyp2} and identity~\eqref{alpha} yields
\begin{equation} \label{alphaY}
\forall p,q>0,~\alpha^Y_{p,q}(r)=\underset{r\rightarrow \infty}{o}(r^{-(d+\epsilon) \frac{2+\delta}{\delta}}).
\end{equation}

We deal with the proof in two steps: first, we consider the case of bounded variables and then we extend the result to the more general case.

The first step of the proof follows the approach used by Bolthausen~\cite{Bolt} and Guyon~\cite{Guyon}, while the second step exploits elements from~\cite{Ibragimov}. The main difference lies in the way we deal with the term $A_3$ that appears later on in the proof.

\medskip

\underline{First step: Bounded variables.} Without loss of generality, we consider that $\E[f_i(X\cap C_i)]=0$ for all $i\in\Z^d$. Suppose that we have $\sup_i \|Y_i\|_{\infty}\defeq\sup_i \|f_i(.\cap C_i)\|_{\infty}=M <\infty$ instead of Assumption~\ref{hyp1}. Since $\alpha^Y_{p,q}(r)$ is non increasing in $r$ and is a $o(r^{-d})$ by~\eqref{alphaY}, we can choose a sequence $(r_n)_{n\in\N}$ such that
$$\alpha_{p,q}^Y(r_n)\sqrt{|I_n|}\rightarrow 0~~\mbox{and}~~r_n^{-d}\sqrt{|I_n|}\rightarrow \infty.$$
For $i\in\Z^d$, define
$$S_{i,n}=\sum_{\substack{j\in I_n \\ \lun{i-j}\leq r_n}} Y_j,\quad S^*_{i,n}=S_n-S_{i,n},\quad a_n=\sum_{i\in I_n}\E[Y_iS_{i,n}],\quad \bar{S}_n=\frac{1}{\sqrt{a_n}}S_n,\quad \bar{S}_{i,n}=\frac{1}{\sqrt{a_n}}S_{i,n}.$$
We have $\sigma_n^2=\var(S_n)=a_n+\sum_{i\in I_n}\E[Y_i S_{i,n}^*]$ and, as a consequence of the typical covariance inequality~(\ref{covalpha}) for $\alpha$-mixing random variables, we get
$$\left | \sum_{i\in I_n} \E[Y_i S_{i,n}^*]\right |\leq\hspace{-0.2cm}\sum_{\substack{i,j\in I_n \\ \lun{i-j}>r_n}}\hspace{-0.2cm}|\cov(Y_i,Y_j)|\leq 8 M^2\hspace{-0.2cm}\sum_{\substack{i,j\in I_n \\ \lun{i-j}>r_n}}\hspace{-0.2cm}\alpha^Y_{1,1}(\lun{i-j})
\leq 8 M^2|I_n|\sum_{r>r_n}|\{k\in\Z^d:|k|_1=r\}|\alpha^Y_{1,1}(r).$$
The number of $k\in\Z^d$ satisfying $|k|_1=r$ is bounded by $2(2r+1)^{d-1}$. This is because each of the $d-1$ first coordinates of $k$ takes its values in $\{-r,\cdots,r\}$ and the last coordinate is fixed by the other ones, up to the sign, since $|k|_1=r$. Therefore,
$$\left | \sum_{i\in I_n} \E[Y_iS_{i,n}^*]\right |\leq 16 M^2|I_n|\sum_{r>r_n}(2r+1)^{d-1}\alpha^Y_{1,1}(r).$$
By Assumption~\eqref{alphaY}, this quantity is $o(|I_n|)$ and thus $\sigma_n^2\sim a_n$ as a consequence of Assumption~\ref{hyp3}. We then only need to prove the asymptotic normality of $\widebar{S}_n$. Moreover, since $\sup_n \E[\widebar{S}_n^2]<\infty$ then this will be a consequence of the following condition (see \cite{CLT,Bolt})
$$\lim_{n\rightarrow \infty}\E[(i\lambda-\widebar{S}_n)\exp(i\lambda\widebar{S}_n)]=0,~~ \forall\lambda\in\R.$$
We can split this expression into $(i\lambda-\widebar{S}_n)\exp(i\lambda\widebar{S}_n)=A_1-A_2-A_3$ where
$$\left \{ \begin{tabular}{l} $\displaystyle A_1=i\lambda\exp(i\lambda\widebar{S}_n)\left (1-\frac{1}{a_n}\sum_{j\in I_n}Y_jS_{j,n}\right ),$ \\ 
$\displaystyle A_2=\frac{1}{\sqrt{a_n}}\exp(i\lambda\widebar{S}_n)\sum_{j\in I_n}Y_j\left (1-i\lambda\widebar{S}_{j,n}-\exp(-i\lambda\widebar{S}_{j,n}) \right ),$  \\ 
$\displaystyle A_3=\frac{1}{\sqrt{a_n}}\sum_{j\in I_n}Y_j \exp\left (i\lambda(\widebar{S}_n-\widebar{S}_{j,n})\right ).$ \end{tabular} \right .$$
It was proved by Bolthausen~\cite{Bolt} that $\E[A_1^2]$ and $\E[|A_2|]$ vanish when $n$ goes to infinity if $\sum r^{d-1}\alpha^Y_{p,q}(r)<\infty$ for $p+q\leq 4$ which is the case here. We show that $\E[A_3]$ vanishes at infinity using~(\ref{alpha+}). Notice that we have
$$|\E[A_3]|\leq \frac{|I_n|}{\sqrt{a_n}}\sup_{j\in I_n} \sizeL |\cov\sizeL (f_j(X\cap C_j),\exp\sizeL ( \frac{i\lambda}{\sqrt{a_n}}\sum_{\substack{k\in I_n \\ \lun{k-j}>r_n}} f_k(X\cap C_k) \sizeR )\sizeR )\sizeR |.$$
Define the function 
$$g_j:X\mapsto \exp\sizeL ( \frac{i\lambda}{\sqrt{a_n}}\sum_{\substack{k\in I_n \\ \lun{k-j}>r_n}} f_k(X\cap C_k) \sizeR ).$$
This function is bounded by $1$ and $\F(B_j)$-measurable where $B_j\defeq\bigcup_{k\in I_n,\,  \lun{k-j}>r_n}C_k$ is a bounded set and $\dist(C_j,B_j)$ $\geq(Rr_n-sd)/\sqrt{d}$ (see Lemma~\ref{distcarre}). We have $\|f_j\|_{C_j}\leq 2M$ and for all $X\in\Omega$, for all $x\in B_j$, if we denote by $J_x=\{k: x\in C_k\}$ the set of cubes that contain $x$ then
$$|g_j(X\cup\{x\})-g_j(X)|=\left|1-\exp\left(\frac{i\lambda}{\sqrt{a_n}}\sum_{k\in J_x}(f_k(X\cap C_k\cup\{x\})-f_k(X\cap C_k))\right)\right|\leq\frac{2\lambda M|J_x|}{\sqrt{a_n}}.$$
Lemma~\ref{distcarre} gives us the bound $|J_x|\leq(2sd/R+1)^d$ and thus $\|g_j\|_{B_j}\leq2\lambda M(2sd/R+1)^d/\sqrt{a_n}$. Finally, using Corollary~\ref{mainDPP} we get
\begin{align}
|\E[A_3]|&\leq \frac{4|I_n|s_d}{\sqrt{a_n}}|C_j|\|f_j\|_{C_j}\|g_j\|_{B_j}\int_{\dist(B_j,C_j)}^{\infty}t^{d-1}\sup_{|x-y|\geq t}|D(x,y)|\der t\nonumber\\
&\leq 16s_d M^2\left(\frac{2s^2d}{R}+s\right)^d\lambda\frac{|I_n|}{a_n}\int_{\frac{1}{\sqrt{d}}(Rr_n-sd)}^{\infty} t^{d-1}\sup_{|x-y|\geq t}|D(x,y)|\der t.\label{good}
\end{align}
By assumption~\ref{hyp2} we have that $t^{d-1}\sup_{|x-y|\geq t}|D(x,y)|$ is integrable and by assumption~\ref{hyp3} we have $|I_n|=O(a_n)$ which shows that $\lim_{n\rightarrow \infty} \E[A_3]=0$ concluding the proof of the theorem for bounded variables.

\medskip

\underline{Second step: General Case.} For $N>0$, we define
$$\left \{ \begin{tabular}{l} $S_{1,n}\defeq\sum_{i\in I_n}(F_N(Y_i)-\E[F_N(Y_i)])~~\mbox{where}~~F_N:x\mapsto x\cara{|x|\leq N},$ \\ 
$S_{2,n}:=\sum_{i\in I_n}(\tilde{F}_N(Y_i)-\E[\tilde{F}_N(Y_i)])~~\mbox{where}~~\tilde{F}_N:x\mapsto x\cara{|x|>N}.$ \end{tabular} \right .$$
Let $\sigma^2_n(N)\defeq\var(S_{1,n})$, from the first step of the proof we have $\sigma_n(N)^{-1}S_{1,n}\overset{\mathcal{L}}{\longrightarrow}\mathcal{N}(0,1)$. Let $1>\gamma>(1+\frac{\epsilon}{d}(1+\frac{\delta}{2}))^{-1}$ and define $C_N\defeq\sup_{i} \|Y_i\cara{|Y_i|>N}\|_{2+\delta\gamma}$. By assumption~\ref{hyp1} we have that $C_N$ vanishes when $N\rightarrow \infty$ and by assumption~\ref{hyp3} we have that $|I_n|\leq c\sigma_n^2$ for a sufficiently large $n$, where $c$ is a positive constant. By~(\ref{covalpha}),
\begin{align*}
\frac{1}{\sigma_n^2}\var(S_{2,n})&=\frac{1}{\sigma_n^2}\sum_{i,j\in I_n}\cov(\tilde{F}_N(Y_i),\tilde{F}_N(Y_j))\\
&\leq\frac{|I_n|}{\sigma_n^2}C_N^2\sup_{i\in I_n}\sum_{j\in I_n} 8\alpha^Y_{1,1}(\lun{i-j})^{\frac{\delta\gamma}{2+\delta\gamma}}\\
&\leq\ 16\,c\,C_N^2\sum_{r=0}^{\infty} (2r+1)^{d-1} \alpha^Y_{1,1}(r)^{\frac{\delta\gamma}{2+\delta\gamma}}.
\end{align*}
By assumption~\ref{hyp2} and the choice of $\gamma$ we have $\sum (2r+1)^{d-1}\alpha^Y_{1,1}(r)^{\frac{\delta\gamma}{2+\delta\gamma}}<\infty$ so $\sigma_n^{-1}S_{2,n}$ converges in mean square to $0$  when $N$ goes to infinity, uniformly in $n$. With the same reasoning, we also get the inequality
$$\frac{1}{\sigma_n^2}\left |\cov(S_{1,n},S_{2,n})\right |\leq\ 16\,c\,MC_N\sum_{r=0}^{\infty} (2r+1)^{d-1} \alpha^Y_{1,1}(r)^{\frac{\delta\gamma}{2+\delta\gamma}},$$
where the right hand side tends to 0 when $N$ goes to infinity, uniformly in $n$. Hence $\sigma_n^2(N)$ tends to $\sigma_n^2$ uniformly in $n$ as $N$ goes to infinity. 

Finally, for all constants $\nu>0$ arbitrary small, we can choose $N$ such that $\E[\sigma_n^{-1}|S_{2,n}|]\leq \nu$ and $|1-\sigma_n(N)/\sigma_n|\leq \nu$ for all $n$ sufficiently large. By looking at the characteristic function of $\sigma_n^{-1}S_n$ we get
\begin{align*}
\left |\E\left (e^{\frac{ixS_n}{\sigma_n}}\right )-e^{-\frac{1}{2}x^2} \right |
&\leq \left |\E\left (e^{\frac{ixS_{1,n}}{\sigma_n}}\right )-\E\left (e^{\frac{ixS_{1,n}}{\sigma_n(N)}}\right ) \right |
+\left |\E\left (e^{\frac{ixS_{1,n}}{\sigma_n(N)}}\right )-e^{-\frac{1}{2}x^2} \right |
+\left |\E\left (e^{\frac{ixS_{2,n}}{\sigma_n}}-1\right ) \right |\\
&\leq x\E\left(\left |\frac{S_{1,n}}{\sigma_n(N)}\right |\right)\left|1-\frac{\sigma_n(N)}{\sigma_n}\right|+o(1)+x \nu\\
&\leq 2x \nu+o(1)
\end{align*}
concluding the proof.
\end{proof}

\section{Application to determinantal point processes} \label{sec:DPP}
In this section, we give a CLT for a wide class of functionals of DPPs. This result is a key tool for the asymptotic inference of DPPs. As an application treated in Section~\ref{waag}, we get the consistency and the asymptotic normality of the two-step estimation method of~\cite{waag} for a parametric inhomogeneous DPP.

\subsection{Negative association and \texorpdfstring{$\alpha$}{alpha}-mixing for DPPs}
We recall that a DPP $X$ on $\R^d$ is defined trough its intensity functions with respect to the Lebesgue measure that must satisfy
$$\forall n\in\N,~\forall x\in (\R^d)^n,~\rho_n(x_1,\cdots,x_n)=\det(K[x])~~\mbox{with}~K[x]\defeq(K(x_i,x_j))_{i,j\in \{1,\cdots,n\}}.$$
The function $K:(\R^d)^2\rightarrow\mathbb{C}$ is called the kernel of $X$ and is assumed to satisfy the following standard general condition ensuring the existence of $X$.

\medskip

\noindent {\bf Condition $\mathcal H$}: The function $K:(\R^d)^2\rightarrow\mathbb{C}$ is a locally square integrable hermitian measurable function such that its associated integral operator $\mathcal{K}$ is locally of trace class with eigenvalues in $[0,1]$. 

\medskip

\noindent  This condition is not necessary for existence, in particular there are examples of DPPs having a non-hermitian kernel. It is nonetheless very general and is assumed in most studies on DPPs. Basic properties of DPPs can be found in~\cite{Hough,Shirai,Lyons}. In particular, from \cite[Theorem 1.4]{NA} and~\cite[Theorem 3.7]{Lyons}, we know that DPPs are NA.

\begin{theo}[\cite{NA,Lyons}]
Let $K$ satisfy Condition $\mathcal H$, then a DPP with kernel $K$ is NA.
\end{theo}
By definition, for a DPP with kernel $K$ we have $D(x,y)=-|K(x,y)|^2$ where $D$ is introduced in \eqref{defD}.  
Hence, using the last theorem and Proposition~\ref{sidePP} we get the following strong mixing coefficients of a DPP, where we define
\begin{equation}\label{defomega}\omega(r)\defeq\sup_{|x-y|\geq r}|K(x,y)|.\end{equation}
\begin{corr} \label{sideDPP}
Let $X$ be a DPP with kernel $K$ satisfying $\mathcal H$. Then, for all $p,q>0$,
\begin{equation} \label{alphaDPP}
\left \{ \begin{tabular}{l} $\alpha_{p,q}(r) \leq \underset{\substack{|A|<p,|B|<q \\ \rm{dist}(A,B)>r}}{\sup}\int_{A\times B} |K(x,y)|^2 \leq pq\omega(r)^2$, \\  $\alpha_{p,\infty}(r)\leq p\area_d\int_r^{\infty} \omega^2(t)t^{d-1}\der t$. \end{tabular} \right .
\end{equation}
\end{corr}

It is worth noticing that this result, and so the covariance inequality~\eqref{NAbound}, is optimal in the sense that for a wide class of DPPs, the $\alpha$-mixing coefficient $\alpha_{p,q}(r)$ do not decay faster than $\sup_{\substack{|A|<p,|B|<q \\ \rm{dist}(A,B)>r}}|\cov(N(A),N(B))|$ when $r$ goes to infinity, as stated in the following proposition.

\begin{prop}
Let $X$ be a DPP with kernel $K$ satisfying $\mathcal H$. We further assume that $K$ is bounded, takes its values in $\R_+$ and is such that $\|\mathcal{K}\|<1$ where $\|.\|$ is the operator norm. Then, for all $p,q,r>0$,
\begin{equation}\label{optim}
(1-\|\mathcal{K}\|)^{\frac{(p+q) \|K\|_{\infty}}{\|\mathcal{K}\|}}\sup_{\substack{|A|<p,|B|<q \\ \rm{dist}(A,B)>r}}\int_{A\times B} |K(x,y)|^2\leq\alpha_{p,q}(r)\leq \sup_{\substack{|A|<p,|B|<q \\ \rm{dist}(A,B)>r}}\int_{A\times B} |K(x,y)|^2.
\end{equation}
\end{prop}
\begin{proof}
The upper bound for $\alpha_{p,q}(r)$ is just the one in \eqref{alphaDPP}.
The lower bound is obtained through void probabilities. Let $p,q,r>0$ and $A,B\subset\R^d$ such that $|A|<p$, $|B|<q$ and $\dist(A,B)>r$. By definition, for any such sets $A$ and $B$,  $\alpha_{p,q}(r)\geq |\mathbb{P}(N(A)=0)\mathbb{P}(N(B)=0)-\mathbb{P}(N(A\cup B)=0)|$. The void probabilities of DPPs are known (see~\cite{Shirai}) and equal to
$$\mathbb{P}(N(A)=0)=\exp\left(-\sum_{n\geq 1}\frac{\tr(\mathcal{K}_A^n)}{n}\right)$$
where $\mathcal{K}_A$ is the projection of $\mathcal{K}$ on the set of square integrable functions $f:A\rightarrow\R$. Moreover, $\mathbb{P}(N(A)=0)\mathbb{P}(N(B)=0)-\mathbb{P}(N(A\cup B)=0)\geq 0$  by negative association, and we have 
\begin{align}
&\mathbb{P}(N(A)=0)\mathbb{P}(N(B)=0)-\mathbb{P}(N(A\cup B)=0)\nonumber\\
&=\exp\left(-\sum_{n\geq 1}\frac{\tr(\mathcal{K}_{A\cup B}^n)}{n}\right)\left(\exp\left(\sum_{n\geq 1}\frac{\tr(\mathcal{K}_{A\cup B}^n)-\tr(\mathcal{K}_A^n)-\tr(\mathcal{K}_B^n)}{n}\right)-1\right)\nonumber\\
&\geq\exp\left(-\sum_{n\geq 1}\frac{\tr(\mathcal{K}_{A\cup B}^n)}{n}\right)\sum_{n\geq 1}\frac{\tr(\mathcal{K}_{A\cup B}^n)-\tr(\mathcal{K}_A^n)-\tr(\mathcal{K}_B^n)}{n}.\label{min1}
\end{align}
Using the classical trace inequality we get
$$\tr(\mathcal{K}_{A\cup B}^n)\leq\|\mathcal{K}_{A\cup B}\|^{n-1}\tr(\mathcal{K}_{A\cup B})\leq\|\mathcal{K}\|^{n-1}\int_{A\cup B}K(x,x)\der x\leq \|\mathcal{K}\|^{n-1}(p+q)\|K\|_{\infty},$$
thus
\begin{equation}\label{min2}
\exp\left(-\sum_{n\geq 1}\frac{\tr(\mathcal{K}_{A\cup B}^n)}{n}\right)\geq (1-\|\mathcal{K}\|)^{\frac{(p+q)\|K\|_{\infty}}{\|\mathcal{K}\|}}.
\end{equation}
Moreover, since $A$ and $B$ are disjoint sets, we can write
\begin{multline}\label{min3}
\tr(\mathcal{K}_{A\cup B}^n)-\tr(\mathcal{K}_A^n)-\tr(\mathcal{K}_B^n)=\int_{(A\cup B)^n} K(x_1,x_2)\cdots K(x_{n-1},x_n)K(x_n,x_1)\der x_1\cdots\der x_n
\\-\int_{A^n\cup B^n} K(x_1,x_2)\cdots K(x_{n-1},x_n)K(x_n,x_1)\der x_1\cdots\der x_n,
\end{multline}
which vanishes when $n=1$, is equal to $2\int_{A\times B} |K(x,y)|^2$ when $n=2$ and is non-negative for $n\geq 3$ since $K$ is assumed to be non-negative. Finally, by combining~(\ref{min1}),~(\ref{min2}) and~(\ref{min3}) we get the lower bound in~(\ref{optim}).
\end{proof}

\subsection{Central limit theorem for functionals of DPPs} \label{sec:CLTDPP}
We investigate the asymptotic distribution of functions that can be written as a sum over subsets of close enough points of $X$, namely
\begin{equation} \label{statform}
f(X):=\sum_{S \subset X} f_0(S),
\end{equation}
where $f_0$ is a bounded function vanishing when $\diam(S)>\tau$ for a certain fixed constant $\tau>0$. 
The typical example, encountered in asymptotic inference, concerns functions $f_0$ that are supported on sets $S$ having exactly $p$ elements, in which case \eqref{statform} often takes the form
\begin{equation}\label{p-uplet}f(X)=\frac{1}{p!}\sum_{x_1,\cdots,x_p\in X}^{\neq}f_0(x_1,\cdots,x_p),\end{equation}
where the sum is done over ordered $p$-tuples  of $X$ and the symbol $\neq$ means that we consider  distinct points. 
The asymptotic distribution of \eqref{p-uplet} has been investigated in~\cite{Sosh} when $p=1$ and in \cite{BiscioMix} for general $p$ and stationary DPPs. 

In the next theorem, we extend these settings to functionals like \eqref{statform} applied to general non stationary DPPs. Some discussion and comments are provided after its proof.
We use Minkwoski's notation and write $A\oplus r$ for the set $\bigcup_{x\in A}\mathcal{B}(x,r)$.

\begin{theo} \label{CLTDPP}
Let $X$ be a DPP associated to a kernel $K$ that satisfies $\mathcal H$ and that is further bounded. Let $\tau>0$ and $f:\Omega\rightarrow\R$ be a function of the form
$$f(X):=\sum_{S \subset X} f_0(S)$$
where $f_0$ is a bounded function vanishing when $\diam(S)>\tau$. Let $(W_n)_{n\in\N}$ be a sequence of increasing subsets of $\R^d$ such that $|W_n|\rightarrow \infty$ and let $\sigma_n^2:=\var(f(X\cap W_n))$. Assume that there exists $\epsilon>0$ and $\nu>0$ such that the following conditions are satisfied:
\begin{enumerate}[label=(H\arabic*)]
\item $|\partial W_n\oplus(\tau+\nu)|=o(\vol{W_n})$;\label{hyp2.2}
\item $\omega(r)=o(r^{-(d+\epsilon)/2})$; \label{hyp2.3}
\item $\liminf_{n}\vol{W_n}^{-1}\sigma_n^2>0$. \label{hyp2.4}
\end{enumerate}
Then,
$$\frac{1}{\sigma_n}(f(X\cap W_n)-\E[f(X\cap W_n)])\cvlaw\mathcal{N}(0,1).$$
\end{theo}

\begin{proof}
In order to apply Theorem~\ref{CLT}, we would like to rewrite $f$ as a sum over  cubes of a lattice.  Unfortunately,  for disjoint sets $A,B\subset \R^d$, $f(X\cap A)+f(X\cap B)\neq f(X\cap (A\cup B))$ in general. Instead, we apply Theorem~\ref{CLT} to an auxiliary function, close to $f$,  as follows.
Define $S^0$ as the barycentre of the set $S$.  We write
\begin{equation} \label{statform2}
f_W(X)=\sum_{S\subset X} f_0(S)\cara{W}(S^0)
\end{equation}
for the sum over the subsets of points of $X$ with barycentre in $W\subset\R^d$.
Now, we split $\R^d$ into little cubes the following way. Let $C_0$ be a given $d$-dimensional cube with a given side-length $0<s\leq \nu/\sqrt{d}$. For all $i\in\Z^d$, let $C_i$ be the translation of $C_0$ by the vector $s\cdot i$. Let $I_n\defeq\{i: C_i\oplus\tau\subset W_n\}$ and $\widetilde{W}_n=\bigcup_{i\in I_n}C_i$. An illustration of these definitions is provided in Figure~\ref{dessin}. Since $f_{\tilde{W}_n}(X)=\sum_{i\in I_n}f_{C_i}(X)$ and each $f_{C_i}$ are $\F(C_i\oplus \tau)$-measurable then $f_{\tilde{W}_n}$ is the ideal candidate to use Theorem~\ref{CLT} on. Thus, we first prove that the difference between $f_{\tilde{W}_n}$ and $f(X\cap W_n)$ is asymptotically negligible and then that $f_{\tilde{W}_n}$ satisfies the conditions of Theorem~\ref{CLT}.

\begin{figure}[h]
   \includegraphics[width=7cm]{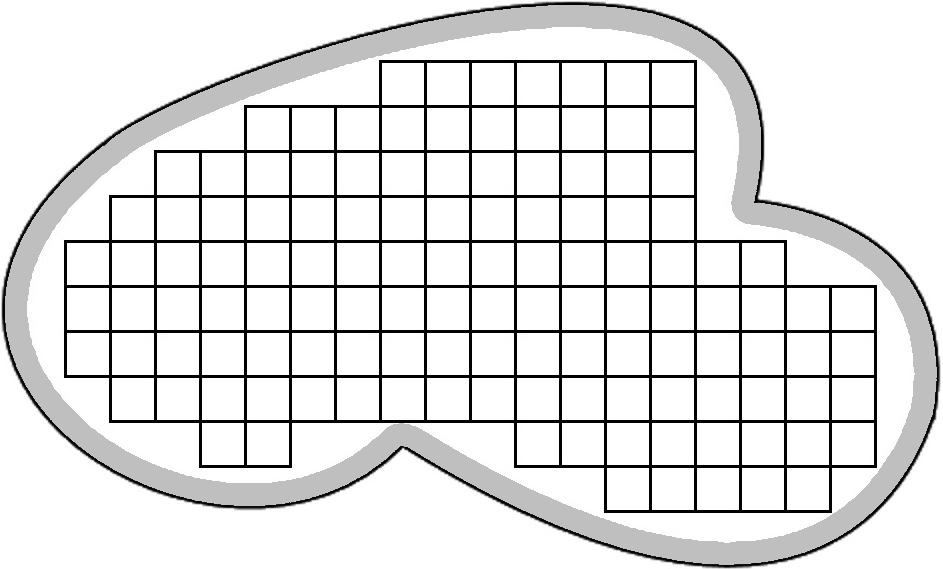}
	   \caption{\label{dessin} Example of illustration of the definition of $\tilde{W}_n$. Here, the black border is $\partial W_n$, the grey area corresponds to $(\partial W_n\oplus\tau)\cap W_n$ and the square lattice corresponds to $\tilde{W}_n$.}
\end{figure}

First of all, notice that $\dist(C_i,\partial W_n)\geq\tau$ for all $i\in I_n$. 
Therefore, for any point in $W_n$ at a distance greater than $\tau+s\sqrt{d}$ from $\partial W_n$, the cube $C_i$  of side-length $s$ containing it is at a distance at least $\tau$ from $\partial W_n$, hence it is one of the $C_i$ in $\tilde{W}_n$ and we get
$$\vol{W_n\backslash\tilde{W}_n}\leq\vol{\partial W_n\oplus (\tau+s\sqrt{d})}.$$
Hence, by Assumption~\ref{hyp2.2}, $\vol{W_n}\sim\vol{\tilde{W}_n}$. Now,
\begin{equation} \label{Diff}
f(X\cap W_n)-f_{\tilde{W}_n}(X)=\sum_{S\subset X\cap W_n} f_0(S)\cara{W_n\backslash\tilde{W}_n}(S^0).
\end{equation}
Since $f_0$ vanishes when two points of $S$ are at distance further than $\tau$, then the sum in~(\ref{Diff}) only concerns the subsets $S$ of $X\cap((W_n\backslash\tilde{W}_n)\oplus \tau)\cap W_n)$. By Lemma~\ref{ugly}, the variance of $f(X\cap W_n)-f_{\tilde{W}_n}(X)$ is then a $O(\vol{(W_n\backslash\tilde{W}_n)\oplus \tau})$,  whence a $o(\vol{W_n})$ and finally a $o(\sigma_n^2)$ by Assumption~\ref{hyp2.4}. Therefore, $\sigma_n^{-1}(f(X\cap W_n)-\E[f(X\cap W_n)])$ has the same limiting distribution as $\sigma_n^{-1}(f_{\tilde{W}_n}(X)-\E[f_{\tilde{W}_n}(X)])$. Moreover, we have
$$|\cov(f_{\tilde{W}_n}(X),f(X\cap W_n)-f_{\tilde{W}_n}(X))|\leq\sigma_n\sqrt{\var(f(X\cap W_n)-f_{\tilde{W}_n}(X))}=\sigma_n o(\sqrt{|W_n|})=o(\sigma_n^2)$$
by Assumptions~\ref{hyp2.2},~\ref{hyp2.4} and Lemma~\ref{ugly} proving that $\sigma_n^{-1}(f(X\cap W_n)-\E[f(X\cap W_n)])$ has the same limiting distribution as $\var(f_{\tilde{W}_n}(X))^{-1/2}(f_{\tilde{W}_n}(X)-\E[f_{\tilde{W}_n}(X)])$.

We conclude by showing that the random variables $Y_i=f_{C_i}(X)-\E[f_{C_i}(X)]$ satisfy the assumptions of Theorem~\ref{CLT}. A rough bound on $f$ gives us $|f_{C_i}(X)|\leq \|f_0\|_{\infty} 2^{N(C_i\oplus \tau)}$ so, by Lemma~\ref{bornenul}, 
$$\forall n\in\N,~~\sup_{i\in\Z^d}\E[|Y_i|^n]<\infty.$$
This means that the $Y_i$'s satisfy Assumption~\ref{hyp1} for all $\delta>0$ and thus~\ref{hyp2} as a consequence of~\ref{hyp2.3}. Finally, since $|I_n|=s^{-d}\vol{\tilde{W}_n}=O(\vol{W_n})$ and $\var(f_{\tilde{W}_n}(X))\sim\sigma_n^2$, we have 
$$\liminf_{n}|I_n|^{-1}\var(f_{\tilde{W}_n}(X))>0$$
by Assumption~\ref{hyp2.4}, which concludes the proof of the theorem.
\end{proof}

We highlight some  extensions of this result.
\begin{itemize}
\item[i)] Since the superposition of independent PA (respectively NA) point processes remains a PA (respectively NA) point process, then Theorem~\ref{CLTDPP} holds true for $\alpha$-determinantal point processes where $\alpha\in\{-1/m: m\in\N^*\}$, see~\cite{Shirai} for more information about $\alpha$-DPPs.
\item [ii)]Theorem~\ref{CLTDPP} also extends to $\R^q$-valued functions $f$ where $q\geq 2$. Let $\Sigma_n:=\var(f(X\cap W_n))$. If we replace \ref{hyp2.4} by 
$$\liminf_{n}\vol{W_n}^{-1}\lambda_{\min}(\Sigma_n)>0$$
where $\lambda_{\min}(\Sigma_n)$ denotes the smallest eigenvalue of $\Sigma_n$, then Theorem~\ref{CLTDPP} holds true with the conclusion
$$\Sigma_n^{-1/2}(f(X\cap W_n)-\E[f(X\cap W_n)])\cvlaw\mathcal{N}(0, Id_q)$$
where $Id_q$ is the $q\times q$ identity matrix. Since $\Sigma_n$ does not necessary converge, this result is not a direct application of the Cram\'er-Wold device. Instead, a detailed proof is given in \cite{CLT}.
\item [iii)]In \eqref{statform}, $f_0$ only depends on finite subsets of $\R^d$ and not on the order of the points in each subset. Nonetheless, we can easily extend \eqref{statform} to functions of the form
$$f(X)=\sum_{n\geq 0}\frac{1}{n!}\sum_{x_1,\cdots,x_n\in X}^{\neq} f_0(x_1,\cdots,x_n)$$
where $f_0$ is a bounded function on $\bigcup_{n\geq 0} (\R^d)^n$ that vanishes when two of its coordinates are at a distance greater than $\tau$. Then $f$ still satisfy Theorem \ref{CLTDPP}. This is because we can write
$$f(X)=\sum_{S\subset X}f^{sym}_0(S)$$
where $f^{sym}_0$ is the symmetrization of $f_0$ defined by
$$f^{sym}_0(\{x_1,\cdots,x_n\})\defeq\frac{1}{n!}\sum_{\sigma\in S_n}f_0(x_{\sigma(1)},\cdots,x_{\sigma(n)})$$
where $S_n$ is the symmetric group on $\{1,\cdots,n\}$. Since $f^{sym}_0$ is also bounded and vanishes when $\diam(S)>\tau$ then it satisfies the required assumptions for Theorem \ref{CLTDPP}.
\end{itemize}

Let us comment the assumptions of Theorem \ref{CLTDPP}.
\begin{itemize}
\item Condition \ref{hyp2.2} makes clear the idea that $W_n$ must grow to $\R^d$ as $n\to\infty$, without  being a too irregular set. In the simple case where  $W_n$ is the Cartesian product of intervals, i.e. $W_n=\Delta_n^{(1)}\times\cdots\times \Delta_n^{(d)}$, then \ref{hyp2.2}  is equivalent to $|\Delta_n^{(k)}|\rightarrow \infty$ for all $k$.

\item Condition~\ref{hyp2.3} is not really restrictive and is satisfied by all  classical kernel families. For example, the kernels of the Ginibre ensemble and of the Gaussian unitary ensemble (see~\cite{Hough}) have an exponential decay. Moreover, all translation-invariant kernels used in spatial statistics (see~\cite{Lavancier} and \cite{Biscio}) satisfy $\omega(r)=O(r^{-(d+1)/2})$: the Gaussian and the Laguerre-Gaussian covariance functions have an exponential decay; the Whittle-Matérn and the Cauchy covariance functions satisfy $\omega(r)=o(r^{-d})$; and in the case of the most repulsive DPP in dimension $d$ (as determined in~\cite{Lavancier,Biscio}), which is the slowest decaying Bessel-type kernel,  its kernel is given by
$$K(x,y)=\frac{\sqrt{\rho \Gamma(\frac{d}{2}+1)}}{\pi^{\frac{d}{4}}}\frac{J_{\frac{d}{2}}(2\sqrt{\pi}\Gamma(\frac{d}{2}+1)^\frac{1}{d}\rho^{\frac{1}{d}}||y-x||)}{||y-x||^{\frac{d}{2}}} \Rightarrow \omega(r)=O\left (r^{-\frac{d+1}{2}}\right ),$$
where $\rho>0$ is a constant.  While this DPP satisfies Condition~\ref{hyp2.3}, we point out that its $\alpha$-mixing coefficients decay too slowly to be able to derive a CLT only from them, see the discussion before Theorem~\ref{CLT}. This justifies the 
importance of Condition~\ref{hyp2} in this theorem, obtained by the NA property, and which leads to Condition~\ref{hyp2.3}.

\item Condition~\ref{hyp2.4} is harder to control in the broad setting of Theorem~\ref{CLTDPP}, but we can get sufficient conditions in some particular cases. For example, if $f_0(S)=\cara{|S|=1}$ and $K$ is a translation-invariant continuous kernel then it was shown in~\cite{Sosh} that Condition~\ref{hyp2.4} holds when $K$ is not the Fourier transform of an indicator function. In the peculiar case where $K$ is the Fourier transform of an indicator function, \cite{Sosh} proved that the limiting distribution is still Gaussian but the rate of convergence is different. As another example extending the previous one, assume that  $f_0$ is a non-negative function supported on the set $\{S\subset X: |S|=p\}$ for a given integer $p>0$ and assume that the highest eigenvalue of the integral operator $\mathcal{K}$ associated to $K$ is less than 1. Then, we show in Proposition~\ref{lowerbound} that
$$\liminf_n \frac{1}{|W_n|}\int_{W_n^p}f_0(x)\det(K[x])\der x>0$$
implies~\ref{hyp2.4} and is much easier to verify.
\end{itemize}

\subsection{Application to the two-step estimation of an inhomogeneous DPP} \label{waag}

In this section, we consider DPPs on $\R^2$ with kernel of the form
\begin{equation} \label{Kform}
K_{\beta,\psi}(x,y)=\sqrt{\rho_\beta(x)}C_\psi(y-x)\sqrt{\rho_\beta(y)},\quad\forall x,y\in\R^2,
\end{equation}
where $\beta\in\R^p$ and $\psi\in\R^q$ are two parameters, $C_\psi$ is a correlation function and $\rho_\beta$ is of the form $\rho_\beta(x)=\rho(z(x)\beta^T)$ where $\rho$ is a known positive strictly increasing function and $z$ is a $p$-variate bounded function called covariates. 
This form implies that the first order intensity, corresponding to $\rho_\beta(x)$,  is inhomogeneous and depends on the covariates $z(x)$ through the parameter $\beta$. But all higher order intensity functions once normalized, i.e. $\rho^{(n)}(x_1,\dots,x_n)/(\rho_\beta(x_1)\dots\rho_\beta(x_n))$, are translation-invariant for $n\geq 2$. In particular, the pair correlation (the case $n=2$) is invariant by translation.  This kind of inhomogeneity is sometimes named second-order intensity reweighted stationarity and is frequently assumed in the spatial point process community. 

Existence of DPPs with a kernel like above is for instance ensured if $\rho_\beta(x)$ is bounded by $\rho_{\max}$ and $C_\psi$ is a continuous, square-integrable correlation function on $\R^d$ whose Fourier transform is less than $1/\rho_{\max}$, see~\cite{Lavancier}. For later use, we call $\mathcal H'$ the previous assumptions on $K_{\beta,\psi}$.

Consider the observation of a DPP $X$ with kernel $K_{\beta^*,\psi^*}$, along with the covariates $z$,  within a window $W_n:=[an,bn]\times[cn,dn]$ where $b>a$ and $d>c$. Waagepetersen and Guan~\cite{waag} have proposed the following two-step estimation procedure of $(\beta^*,\psi^*)$ for  second-order intensity reweighted stationary models. First, $\hat{\beta}_n$ is obtained by solving
$$u_{n,1}(\beta):=\displaystyle\sum_{u\in X\cap W_n}\frac{\nabla\rho_\beta(u)}{\rho_\beta(u)}-\int_{W_n}\nabla\rho_\beta(u)\der u=0.$$
where $\nabla\rho_\beta$ denotes the gradient with respect to $\beta$. In the second step, $\hat{\psi}_n$ is obtained by minimizing $m_{n,\hat{\beta}_n}$ where
$$m_{n,\beta}(\psi):=\displaystyle \int_{r_l}^r \left (\left (\sum_{u,v\in X\cap W_n}\frac{\cara{\{0<|u-v|\leq t\}}}{\rho_{\beta}(u)\rho_{\beta}(v)|W_n\cap W_{n,u-v}|}\right )^c-\mathcal K_{\psi}(t)^c\right)^2\der t.$$
Here  $r_l,r$ and $c$ are user-specified non-negative constants, $W_{n,u-v}$ is $W_n$ translated by $u-v$ and $\mathcal K_{\psi}$ is the Ripley $K$-function defined by
$$\mathcal K_{\psi}(t)\defeq\int_{\|u\|\leq t}g_{\psi}(u)\der u$$
where  $g_{\psi}(u)\defeq 1-C_\psi(u)^2/C_\psi(0)^2$ is the pair correlation function of $X$. If we define 
$$u_{n,2}(\beta,\psi) := -|W_n|\frac{\partial m_{n,\beta}(\psi)}{\partial\psi},$$
then the two-step procedure amounts to solve
$$u_n(\beta,\psi):=(u_{n,1}(\beta),u_{n,2}(\beta,\psi))=0.$$

The asymptotic properties of this two-step procedure are established in~\cite{waag}, under various moments and mixing assumptions, with a view to inference for Cox processes. We state hereafter the asymptotic normality of $(\hat\beta_n,\hat\psi_n)$
in the case of DPPs with kernel of the form~\eqref{Kform}. This setting allows us to apply Theorem~\ref{CLTDPP} and  get rid of some restrictive mixing assumptions needed in~\cite{waag}.

The asymptotic covariance matrix of $(\hat\beta_n,\hat\psi_n)$ depends on two matrices defined  in~\cite[Section 3.1]{waag}, where they are denoted by $\tilde{\Sigma}_n$ and $I_n$. We do not reproduce their expression, which is hardly tractable. An assumption in~\cite{waag} ensures the asymptotic non-degeneracy of this covariance matrix and we also need this assumption in our case, see \ref{N3} below. Unfortunately, as discussed in~\cite{waag},  it is hard to check this assumption for a given model, particularly because it depends on the covariates $z$. We are confronted by the same limitation in our setting. On the other hand, the other assumptions of the following theorem are not restrictive. In particular almost all standard kernels satisfy \ref{RegC} below, see the discussion after Theorem~\ref{CLTDPP}.

\begin{theo}
Let $X$ be a DPP with kernel $K_{\beta^*,\psi^*}$ given by \eqref{Kform} and satisfying $\mathcal H'$. Let $(\hat\beta_n, \hat\psi_n)$ the two-step estimator defined above. We assume the following.
\begin{enumerate} [label=(W\arabic*)]
\item $r_l>0$ if $c<1$; otherwise $r_l\geq 0$,
\item $\rho_\beta$ and $\mathcal K_\psi$ are twice continuously differentiable as functions of $\beta$ and $\psi$,
\item $\sup_{\|x\|\geq r}C_{\psi^*}(x)=O(r^{-1-\epsilon})$, \label{RegC}
\item Condition N3 in~\cite{waag} (concerning the matrices $I_n$ and  $\tilde\Sigma_n$) is satisfied.\label{N3}\end{enumerate}
Then, there exists a sequence $\{(\hat{\beta}_n,\hat{\psi}_n): n\geq 1\}$ for which $u_n(\hat{\beta}_n,\hat{\psi}_n)=0$ with a probability tending to one and 
$$|W_n|^{1/2}[(\hat{\beta}_n,\hat{\psi}_n)-(\beta^*,\psi^*)]I_n\tilde{\Sigma}_n^{-1/2}\cvlaw\mathcal{N}(0,Id).$$
\end{theo}
\begin{proof}
Let $\rho_k$ be the $k$th intensity function of the DPP with kernel $(x,y)\mapsto C_{\psi^*}(y-x)$. In order to apply Theorem 1 in~\cite{waag} we need to show that
\begin{enumerate}
\item $\rho_2,\rho_3$ are bounded and there is a constant $M$ such that for all $u_1,u_2\in \R^2$, $\int|\rho_3(0,v,v+u_1)-\rho_1(0)\rho_2(0,u_1)|\der v<M$ and $\int|\rho_4(0,u_1,v,v+u_2)-\rho_2(0,u_1)\rho_2(0,u_2)|\der v<M$, \label{hyp3.1}
\item $\|\rho_{4+2\delta}\|_{\infty}<\infty$ for some $\delta>0$, \label{hyp3.2}
\item $\alpha_{a,\infty}(r)=O(r^{-d})$ for some $a>8r^2$ and $d>2(2+\delta)/\delta$. \label{hyp3.3}
\end{enumerate}
The first property (\ref{hyp3.1}) is a consequence of~\ref{RegC}. This is because we can write
$$|\rho_3(0,v,v+u_1)-\rho_1(0)\rho_2(0,u_1)|=|2C_{\psi^*}(v)C_{\psi^*}(u_1)C_{\psi^*}(v+u_1)-C_{\psi^*}(0)(C_{\psi^*}(v+u_1)^2+C_{\psi^*}(v)^2)|$$
which is  bounded by $2|C_{\psi^*}(0)|(C_{\psi^*}(v+u_1)^2+C_{\psi^*}(v)^2)$ and
$$\int_{\R^2} C_{\psi^*}(v)^2\der v\leq 2\pi\int_0^{\infty}r\sup_{\|x\|=r}|C_{\psi^*}(x)|^2\der r$$
which is finite by Assumption~\ref{RegC}. The term $\rho_4(0,u_1,v,v+u_2)-\rho_2(0,u_1)\rho_2(0,u_2)$ can be treated the same way.
For a DPP, (\ref{hyp3.2}) is satisfied for any $\delta>0$. Finally, (\ref{hyp3.3}) is the one that causes an issue since, as stated before, the $\alpha$-mixing coefficient we get in Corollary~\ref{sideDPP} decreases slower than what we desire. But, the only place  this assumption is used in~\cite{waag} is to prove the asymptotic normality of their estimator in their Lemma~5, which can also be derived as a consequence of our Theorem~\ref{CLTDPP} with Assumption~\ref{RegC}. 
\end{proof}

\appendix

\section{Proof of Theorem~\ref{trueNA}} \label{proofAss}
We use the following variant of the monotone class theorem (see~\cite[Theorem 22.1]{mono}).
\begin{theo} \label{Dyn}
Let $\mathcal{S}$ be a set of bounded functions stable by bounded monotone convergence and uniform convergence. Let $\mathcal{C}$ be a subspace of $\mathcal{S}$ such that $\mathcal{C}$ is an algebra containing the constant function $\tilde{1}$. Then, $\mathcal{S}$ contains all bounded functions measurable over $\sigma(\mathcal{C})$.
\end{theo} 
Now, let $A,B_1,\cdots,B_k$ be pairwise distinct Borel subsets of $\R^d$ and $g:\N^k\mapsto\R$ be a coordinate-wise increasing function. We denote by $\Omega_A$ the set of locally finite point configurations in $A$ and we define $\mathcal{S}$ as the set of functions $f:\Omega_A\mapsto\R$ such that
\begin{equation}\label{www}
\E[\hat{f}(X\cap A)g(N(B_1),\cdots,N(B_k))]\leq\E[\hat{f}(X\cap A)]\E[g(N(B_1),\cdots,N(B_k)],
\end{equation}
where $\hat{f}(X)\defeq\sup_{Y\leq X}f(Y)$.
Note that $\hat{f}$ is an increasing function and that $f$ is increasing iff $\hat{f}=f$. Our goal is to prove that $\mathcal{S}$ contains all bounded functions supported over $A$. Because of the definition of NA point processes~(\ref{NAvec}), we know that $\mathcal{S}$ contains the set $\mathcal{C}$ of functions of the form $f(N(A_1),\cdots,N(A_k))$ where the $A_i$ are pairwise disjoints Borel subsets of $A$. In particular, since point processes over $A$ are generated by the set of random vectors $\{(N(A_1),\cdots,N(A_k)): A_i\subset A~\mbox{disjoints}, k\in\N\}$, then we only need to verify that $\mathcal{S}$ and $\mathcal{C}$ satisfy the hypothesis of Theorem~\ref{Dyn} to conclude.
\begin{itemize}
\item \underline{Stability of $\mathcal{S}$ by bounded monotonic convergence:}  Since~(\ref{www}) is invariant if we add a constant to $f$ and $f$ is bounded then we can consider $f$ to be positive. Now, notice that for all functions $h$ and $k$, $h\leq k \Rightarrow \hat{h}\leq\hat{k}$ and $h\geq k \Rightarrow \hat{h}\geq\hat{k}$. So, if we take a positive bounded monotonic sequence $f_n\in\mathcal{S}$ that converges to a bounded function $f$, then $\hat{f_n}$ is also a positive bounded monotonic sequence that consequently  converges to a function $g$. Suppose that $(f_n)_n$ is an increasing sequence (the decreasing case can be treated similarly) and let us show that $g=\hat{f}$. Let $X\in\Omega_A$, for all $Y\subset X$, $f_n(Y)\leq f(Y)$. Taking the supremum then the limit gives us $g(X)\leq \hat{f}(X)$. Moreover, for all $Y\subset X$, $g(X)\geq \hat{f_n}(X)\geq f_n(Y)$. Taking the limit gives us that $g(X)\geq f(Y)$ for all $Y\subset X$ so $g(X)\geq \hat{f}(X)$ which proves that $g=\hat{f}$. Using the monotone convergence theorem we conclude that
\begin{align} \label{www2}
&\E[\hat{f_n}(X\cap A)]\rightarrow\E[\hat{f}(X\cap A)] \nonumber \\
\mbox{and} \quad&\E[\hat{f_n}(X\cap A)g(N(B_1),\cdots,N(B_k))]\rightarrow\E[\hat{f}(X\cap A)g(N(B_1),\cdots,N(B_k))],
\end{align}
which proves that $f\in\mathcal{S}$

\item \underline{Stability of $\mathcal{S}$ by uniform convergence:}  Let $f_n$ be a sequence over $\mathcal{S}$ converging uniformly to a function $f$ then, by Lemma~\ref{lipsup}, $\hat{f}_n$ also converges uniformly (and therefore in $L^1$) to $\hat{f}$. As a consequence,~(\ref{www2}) is also satisfied in this case so $f\in\mathcal{S}$.

\item \underline{$\mathcal{C}$ is an algebra:} It is easily shown that $\mathcal{C}$ is a linear space containing $\tilde{1}$ so we only need to prove that $\mathcal{C}$ is stable by multiplication. Let $A_1,\cdots,A_r$ and $A'_1\cdots,A'_s$ be two sequences of pairwise distinct Borel subsets of $A$. Let $f=f(N(A_1),\cdots,N(A_r))\in\mathcal{C}$ and $h=h(N(A'_1),\cdots,N(A'_r))\in\mathcal{C}$. We can write
$$N(A_i)=N(A_i\backslash\cup_j A'_j)+\sum_j N(A_i\cap A'_j)~~\mbox{and}~~N(A'_i)=N(A'_i\backslash\cup_j A_j)+\sum_j N(A'_i\cap A_j),$$
so $f\cdot h$ can be expressed as a function of the number of points in the subsets $A_i\backslash\cup_j A'_j$, $A'_i\backslash\cup_j A_j$ and $A_i\cap A'_j$ that are all pairwise distinct Borel subsets of $A$, proving that $\mathcal{C}$ is stable by multiplication.
\end{itemize}

This concludes the proof that $\mathcal{S}$ contains all bounded functions supported over $\Omega_A$. By doing the same exact reasoning on the set of bounded functions $g$ satisfying $\E[f(X\cap A)g(X\cap B)]\leq\E[f(X\cap A)]\E[g(X\cap B)]$ for a fixed $f$ we  obtain the same result which concludes the proof.

\section{Auxiliary results}

\begin{lem} \label{lipsup}
Let $E$ be a set and $f,g:E\rightarrow\R$ be two functions, then
$$\left |\sup_{x\in E}f(x)-\sup_{y\in E}g(y)\right | \leq \|f-g\|_{\infty}$$
\end{lem}
\begin{proof}
The proposition becomes trivial once we write
$$f(x)\leq g(x)+\|f-g\|_{\infty}\leq \sup_{y\in E}g(y)+\|f-g\|_{\infty}$$
Taking the supremum yields the first inequality. Moreover, by symmetry of $f$ and $g$ the second one follows similarly.
\end{proof}

\begin{lem} \label{distcarre}
Let $i,j\in\Z^d$ such that $\lun{i-j}\defeq\sum_{l=1}^d |i_l-j_l|=r$. Let $s,R>0$ and $C_i,C_j$ be the $d$-dimensional cubes with side length $s$ and respective centre $x_i=R\cdot i$ and $x_j=R\cdot j$. Then,
$$\rm{dist}(C_i,C_j)\geq \frac{1}{\sqrt{d}}(rR-sd).$$
Moreover, each cube intersects at most $(2sd/R)^d$ other cubes with centers on $R\cdot\Z^d$ and side length $s$.
\end{lem}
\begin{proof}
Since each point of a $d$-dimensional square with side length $s$ is at distance at most $s\sqrt{d}/2$ from its centre, we get $\dist(C_i,C_j)\geq\sqrt{(Ri_1-Rj_1)^2+\cdots+(Ri_d-Rj_d)^2}-s\sqrt{d}$ which takes its minimum when $|i_l-j_l|=r/d$ for all $1\leq l\leq d$ hence $\dist(C_i,C_j)\geq rR/\sqrt{d}-s\sqrt{d}.$\\
In particular, if $\lun{i-j}>sd/R$ then $C_i\cap C_j=\emptyset$, hence for all $i\in\Z^d$
$$|\{j\in\Z^d: C_i\cap C_j\neq\emptyset, i\neq j\}|\leq |\{j:0<\lun{i-j}\leq sd/R\}|\leq\left (\frac{2sd}{R}+1\right)^d$$
\end{proof}

\begin{lem} \label{ineqmatrix}
Let $M$ and $N$ be two $n\times n$ semi-positive definite matrices such that $0\leq M\leq N^{-1}$ where $\leq$ denotes the Loewner order. Then,
$$\det(Id-MN)\geq 1-\tr(MN)$$
\end{lem}
\begin{proof}
First, let us consider the case where $N=Id$.
If $\tr(M)\geq 1$ then $\det(Id-M)\geq 0\geq 1-\tr(M)$. Otherwise, we denote by $\Sp(M)$ the spectrum of $M$ and since all eigenvalues are in [0,1[, we can write
\begin{align}
\det(Id-M)&=\prod_{\lambda\in\Sp(M)}\exp(\ln(1-\lambda))\nonumber\\
&=\prod_{\lambda\in\Sp(M)}\exp{\left ( -\sum_{n=0}^{\infty}\frac{\lambda^n}{n} \right )}\nonumber\\
&=\exp{\left ( -\sum_{n=0}^{\infty}\frac{\tr(M^n)}{n} \right )}\nonumber\\
&\geq\exp{\left ( -\sum_{n=0}^{\infty}\frac{\tr(M)^n}{n} \right )}\nonumber\\
&=\exp(\log(1-\tr(M)))\nonumber\\
&=1-\tr(M).\label{resultmatrix}
\end{align}
Getting back to the general case, we can write $N$ as $S^TS$ and by Sylvester's determinant identity we get that $\det(Id-MN)=\det(Id-SMS^T)$. Since we assumed that $0\leq M\leq N^{-1}$ then $0\leq SMS^T\leq Id$ and by applying~(\ref{resultmatrix}) this concludes the proof:
$$\det(Id-MN)=\det(Id-SMS^T)\geq 1-\tr(SMS^T)=1-\tr(MN).$$
\end{proof}

\begin{prop} \label{lemdet}
Let $M$ be a $n \times n$ semi-definite positive matrix of the form $M=\begin{pmatrix} M_1 & N \\ N^T & M_2 \end{pmatrix}$ where $M_1$ is a $k \times k$ semi-definite positive matrix, $M_2$ is a $(n-k) \times (n-k)$ semi-definite positive matrix and $N$ is a $k\times (n-k)$ matrix. We define $||A||_{\infty}\defeq\sup |a_{i,j}|$ for any matrix $A$. Then,
$$0\leq \det(M_1)\det(M_2)-\det(M) \leq k(n-k)\tr(N^TN)||M||_{\infty}^{n-2}.$$
\end{prop}
\begin{proof}
First, we assume that $M_1$ and $M_2$ are invertible. Using Schur's complement, we can write
$$\det(M)=\det(M_1)\det(M_2)\det(Id-M_1^{-1}N^TM_2^{-1}N)$$
where $0\leq N^TM_2^{-1}N \leq M_1$ with $\leq$ being the Loewner order. $N^TM_2^{-1}N$ being semi-definite positive implies
$$\det(M)\leq\det(M_1)\det(M_2),$$
while the inequality $N^TM_2^{-1}N \leq M_1$ gives us (see Lemma~\ref{ineqmatrix})
$$\det(M)\geq\det(M_1)\det(M_2)(1-\tr(M_1^{-1}N^TM_2^{-1}N)).$$
Therefore,
\begin{multline*}
0\leq \det(M_1)\det(M_2)-\det(M) \leq \tr(\adj(M_1)N^T\adj(M_2)N)\\
\leq\tr(\adj(M_1))\tr(\adj(M_2))\tr(N^TN)=
\sum_{i=1}^k\Delta_i(M_1)\sum_{j=1}^{n-k}\Delta_j(M_2)\tr(N^TN),
\end{multline*}
where $\Delta_i(M_1)$ means the $(i,i)$ minor of the matrix $M_1$ and $\adj(M_1)$ is the transpose of the matrix of cofactor of $M_1$. But, since all principal sub-matrices of $M_1$ and $M_2$ are positive definite matrices then their determinant is lower than the product of their diagonal entries, meaning that $\Delta_i(M_1)\leq \prod_{j\neq i}M_1(j,j)\leq||M||_{\infty}^{k-1}$. Doing the same thing for the terms $\Delta_j(M_2)$ gives us the desired result.

If  $M_1$ or $M_2$ is not invertible, a limit argument using the continuity of the determinant leads to the same conclusion.

\end{proof}

\begin{lem} \label{bornenul}
Let $X$ be a DPP with bounded kernel $K$ satisfying $\mathcal H$, $s>0$ and $n>0$, then
$$\sup_{A\subset \R^d,\vol{A}=s}\E[2^{nN(A)}]<\infty$$
\end{lem}
\begin{proof}
Let $n\in\N$ and $A\subset \R^d$ such that $\vol{A}=s$. Since the determinant of a positive semi-definite matrix is always smaller than the product of its diagonal coefficients we get
\begin{align*}
\E[2^{nN(A)}]&=\E\left [\sum_{k=0}^{\infty}\binom{N(A)}{k}(2^n-1)^k \right ]\\
&=\sum_{k=0}^{\infty}\frac{(2^n-1)^k}{k!}\int_{A^k}\det(K[x])\der x\\
&\leq e^{(2^n-1)\|K\|_{\infty}|A|}<\infty.
\end{align*}
\end{proof}

\begin{lem} \label{ugly}
Let $X$ be a DPP on $\R^d$ with bounded kernel $K$ satisfying $\mathcal H$ such that $\omega(r)=O(r^{-\frac{d+\epsilon}{2}})$ for a certain $\epsilon>0$. Then, for all bounded Borel sets $W\subset R^d$ and all bounded functions $g:\bigcup_{p>0}(\R^d)^p\rightarrow\R$ such that $g(S)$ vanishes when $diam(S)>\tau$ for a given constant $\tau>0$,
\begin{equation} \label{tobound}
\var\left(\sum_{S\subset X\cap W} g(S)\right)=O(\vol{W}).
\end{equation}
\end{lem}
\begin{proof} Since $W$ is bounded then $N(W)$ is almost surely finite and we can write
$$\sum_{S\subset X\cap W} g(S)=\sum_{p\geq 0}\sum_{\substack{S\subset X\cap W \\ |S|=p}}g(S)\quad\mbox{a.s.}$$
Looking at the variance of each term individually, we start by developing $\displaystyle\E\left[\left(\sum_{S\subset X\cap W} g(S)\cara{|S|=p}\right)^2\right]$ as
\begin{align}
&\sum_{k=0}^p\E\sizeL[\sum_{\substack{S,T\subset X\cap W \\ |S|=|T|=p,|S\cap T|=k}} g(S)g(T)\sizeR]\\
=&\sum_{k=0}^p\E\sizeL[\sum_{\substack{U\subset X\cap W \\ |U|=2p-k}}\sum_{\substack{S'\subset S\subset U \\ |S'|=k,|S|=p}} g(S)g(S'\cup (U\backslash S))\sizeR]\nonumber\\
=&\sum_{k=0}^p\frac{1}{(2p-k)!}\int_{W^{2p-k}}\sum_{\substack{S'\subset S\subset\{x_1,\cdots,x_{2p-k}\} \\ |S'|=k,|S|=p}} g(S)g(S'\cup (U\backslash S))\rho_{2p-k}(x_1,\cdots,x_{2p-k})\der x_1\cdots\der x_{2p-k}\nonumber\\
=&\sum_{k=0}^p\frac{1}{(2p-k)!}\binom{p}{k}\binom{2p-k}{p}\int_{W^{2p-k}}g(x_1,\cdots,x_p)g(x_1,\cdots,x_k,x_{p+1},\cdots,x_{2p-k})\rho_{2p-k}(x)\der x.\label{midboss}
\end{align}
Since the determinant of a positive semi-definite matrix is smaller than the product of its diagonal terms, we have $|\rho_{2p-k}(x)|\leq\|K\|_{\infty}^{2p-k}$. Moreover, as a consequence of our assumptions on $g$, each term for $k\geq 1$ in~(\ref{midboss}) is bounded by
\begin{align*}
\frac{1}{p!(p-k)!}\binom{p}{k}\int\limits_{W^{2p-k}}\|g\|_{\infty}^2\|K\|_{\infty}^{2p-k}\cara{\{0\leq |x_i-x_1|\leq \tau,~\forall i\}}\der x &\leq\frac{|W|}{p!}\binom{p}{k}\|g\|_{\infty}^2\|K\|_{\infty}^{2p-k}|\mathcal{B}(0,\tau)|^{2p-k-1}\\
&\leq\frac{|W|}{p!}\binom{p}{k}\|g\|_{\infty}^2(1+\|K\|_{\infty})^{2p}(1+|\mathcal{B}(0,\tau)|)^{2p}.
\end{align*}
Hence, 
\begin{equation} \label{bornmoche}
\sum_{k=1}^p\E\sizeL[\sum_{\substack{S,T\subset X\cap W \\ |S|=|T|=p,|S\cap T|=k}} g(S)g(T)\sizeR]\leq |W|\|g\|_{\infty}^2\frac{C_1^p}{p!}
\end{equation}
where $C_1=2(1+\|K\|_{\infty})^2(1+\mathcal{B}(0,\tau))^2$ is a constant independent from $p$ and $W$.
However, even if all terms for $k\geq 1$ in~(\ref{midboss}) are $O(|W|)$, this is not the case of the term for $k=0$ which is a $O(|W|^2)$.
Instead of controlling this term alone, we consider its difference with the remaining term in the variance we are looking at, that is
\begin{multline*}\frac{1}{(p!)^2}\int_{W^{2p}}g(x)g(y)\rho_{2p}(x,y)\der x\der y -  \displaystyle\E\left[\left(\sum_{S\subset X\cap W} g(S)\cara{|S|=p}\right)\right]^2 \\= \frac{1}{(p!)^2}\int_{W^{2p}}g(x)g(y)(\rho_{2p}(x,y)-\rho_{p}(x)\rho_{p}(y))\der x\der y .\end{multline*}
Using Proposition~\ref{lemdet}, we get
$$|\rho_{2p}(x,y)-\rho_{p}(x)\rho_{p}(y)|\leq p^2\|K\|_{\infty}^{2p-2}\sum_{1\leq i,j\leq p}K(x_i,y_j)^2.$$
Now, notice that for all $y\in\R^d$ and $1\leq i\leq p$,
$$\int_{W^p}\cara{\{0<|x_k-x_j|\leq \tau,~\forall j,k\}}|K(x_i,y)|^2\der x\leq|\mathcal{B}(0,\tau)|^{p-1}\int_W |K(x_i,y)|^2\der x_i\leq |\mathcal{B}(0,\tau)|^{p-1}s_d\int_{\R^d}r^{d-1}\omega(r)^2\der r$$
which is finite because of our assumption on $\omega(r)$. Thus, we obtain the inequality
\begin{align}
\int_{W^{2p}}g(x)g(y)|K(x_i,y_j)|^2\der x\der y&\leq\|g\|_{\infty}|\mathcal{B}(0,\tau)|^{p-1}\int_{W^{p+1}}g(x)|K(x_i,y_1)|^2\der x\der y_1\nonumber\\
&\leq|W|\|g\|_{\infty}^2 |\mathcal{B}(0,\tau)|^{2p-2}s_d\int_{\R^d}r^{d-1}\omega(r)^2\der r. \label{intmoche}
\end{align}
By combining~(\ref{bornmoche}) and~(\ref{intmoche}), we get the bound
$$\var\sizeL(\sum_{\substack{S\subset X\cap W \\ |S|=p}} g(S)\sizeR)\leq |W|\|g\|_{\infty}^2\left (\frac{C_1^p}{p!}+\frac{C_2}{p!}\right)$$
where
$$C_2\defeq \left (\sup_{p\geq 0}\frac{p^4\|K\|_{\infty}^{2p-2}|\mathcal{B}(0,\tau)|^{2p-2}}{p!}\right ) s_d\int_{\R^d}r^{d-1}\omega(r)^2\der r$$
is a constant independent from $p$ and $W$. Finally,
$$\sum_{p\geq 0}\var\sizeL(\sum_{\substack{S\subset X\cap W \\ |S|=p}} g(S)\sizeR)=O(|W|)$$
and
$$\sum_{p>q\geq 0}\cov\sizeL(\sum_{\substack{S\subset X\cap W \\ |S|=p}} g(S),\sum_{\substack{S\subset X\cap W \\ |S|=q}} g(S)\sizeR)\leq|W|\|g\|_{\infty}^2\sum_{p,q\geq 0}\sqrt{\left(\frac{C_1^p}{p!}+\frac{C_2}{p!}\right)\left(\frac{C_1^q}{q!}+\frac{C_2}{q!}\right)}=O(|W|)$$
concluding the proof.
\end{proof}

\begin{prop} \label{lowerbound}
Let $p\in\N$, $f:\R^p\rightarrow\R_+$ be a symmetrical measurable function 
and define
$$F(X)=\sum_{\substack{S\subset X \\ |S|=p}}f(S).$$
Let $X$ be a DPP with kernel $K$ satisfying Condition~$\mathcal{H}$ such that $\|\mathcal{K}\|<1$ where $\|\mathcal{K}\|$ is the operator norm of the integral operator associated with $K$. 
If, for a given increasing sequence of compact sets $W_n\subset\R^d$,
\begin{equation} \label{condhyp}
\liminf_n \frac{1}{|W_n|}\int_{W_n^p}f(x)\det(K[x])\der x>0,
\end{equation}
then  $$\liminf_n \frac{1}{|W_n|}\var (F(X\cap W_n))>0.$$
\end{prop}

\begin{proof}
Let $W$ be a compact subset of $\R^d$. The Cauchy-Schwartz inequality gives us
$$\cov(F(X\cap W),N(W))^2\leq\var(F(X\cap W))\var(N(W)).$$
We showed in Lemma~\ref{ugly} that $|W|^{-1}\var(N(W))$ is bounded by a constant $C>0$ so we are only interested in the behaviour of $\cov(F(X\cap W),N(W))$. We start by developing $\E[F(X\cap W)N(W)]$:
\begin{align*}
\E[F(X\cap W)N(W)]&=\E\sizeL [ \sum_{\substack{S\subset X\cap W \\ |S|=p}}f(S)\sum_{x\in X\cap W} 1  \sizeR ]\\
&=\E\sizeL [ \sum_{\substack{S\subset X\cap W \\ |S|=p+1}}\sum_{x\in S}f(S\backslash\{x\}) +p\sum_{\substack{S\subset X\cap W \\ |S|=p}}f(S)  \sizeR ]\\
&=\frac{1}{(p+1)!}\int_{W^{p+1}}\sum_{i=1}^{p+1}f(z\backslash\{z_i\})\det(K[z])\der z+\frac{1}{p!}\int_{W^p}pf(x)\det(K[x])\der x\\
&=\frac{1}{p!}\left(\int_{W^{p}}f(x)\left(p\det(K[x])+\int_W\det(K[x,a])\der a\right)\der x\right).
\end{align*}
We also have
$$\E[F(X\cap W)]\E[N(W)]=\frac{1}{p!}\int_{W^p}f(x)\det(K[x])\der x\int_W K(a,a)\der a,$$
hence
\begin{equation} \label{lowerboundEq1}
\cov(F(X\cap W),N(W))=\frac{1}{p!}\int_{W^p}f(x)\det(K[x])\left( p-\int_W \big (K(a,a)-\det(K[x,a])\det(K[x])^{-1}\big ) \der a \right)\der x.
\end{equation}
Using Schur's complement, we get
\begin{equation} \label{lowerboundEq2}
K(a,a)-\det(K[x,a])\det(K[x])^{-1}=K_{ax}K[x]^{-1}K_{ax}^T
\end{equation}
where we define $K_{ax}$ as the vector $(K(a,x_1),\cdots,K(a,x_p))$. Moreover, since we look at our point process in a compact window $W$, a well-known property of DPPs (see~\cite{Hough}) is that there exists a sequence of eigenvalues $\lambda_i$ in $[0,\|\mathcal{K}\|]$ and an orthonormal basis of $L^2(W)$ of eigenfunctions $\phi_i$ such that
$$K(x,y)=\sum_i \lambda_i\phi_i(x)\bar{\phi}_i(y)~~\forall x,y\in W.$$
As a consequence, $\forall x,y\in W$,
$$\int_W K(x,a)K(a,y) \der a=\sum_i \lambda_i^2\phi_i(x)\bar{\phi}_i(y)$$
which we define as $L(x,y)$. Therefore, for all $x\in W^p$, $L[x]\leq \|\mathcal{K}\|K[x]$ where $\leq$ is the Loewner order for positive definite symmetric matrices and we get
\begin{equation} \label{lowerboundEq3}
\int_W K_{ax}K[x]^{-1}K_{ax}^T \der a=\tr\left(K[x]^{-1}\int_W K_{ax}^TK_{ax} \der a \right)=\tr(K[x]^{-1}L[x])\leq p\|\mathcal{K}\|.
\end{equation}
Finally, since $f$ is non negative, by combining \eqref{lowerboundEq1}, \eqref{lowerboundEq2} and \eqref{lowerboundEq3} we get the lower bound
$$\var(F(X\cap W))\geq\frac{\cov(F(X\cap W),N(W))^2}{\var(N(W))}\geq\frac{(1-\|\mathcal{K}\|)^2}{C(p-1)!^2|W|}\left(\int_{W^{p}}f(x)\det(K[x])\der x\right)^2$$
which proves the proposition.
\end{proof}

\bibliographystyle{plain}
\bibliography{ref}

\begin{thebibliography}{10}

\bibitem{NA_Int}
K.~Alam and K.M.L. Saxena.
\newblock Positive dependence in multivariate distributions.
\newblock {\em Comm. Statist. Theory Methods A}, 10:1183--1196, 1981.

\bibitem{BaHa16Sub}
R.~Bardenet and A.~Hardy.
\newblock Monte {C}arlo with determinantal point processes.
\newblock {\em arXiv preprint arXiv:1605.00361}, 2016.

\bibitem{BiscioMix}
C.A.N. Biscio and F.~Lavancier.
\newblock Brillinger mixing of determinantal point processes and statistical
  applications.
\newblock {\em Electron. J. Statist.}, 10:582--607, 2016.

\bibitem{Biscio}
C.A.N. Biscio and F.~Lavancier.
\newblock Quantifying repulsiveness of determinantal point processes.
\newblock {\em Bernoulli}, 22:2001--2028, 2016.

\bibitem{CLT}
C.A.N. Biscio, A.~Poinas, and R.~Waagepetersen.
\newblock A note on gaps in proofs of central limit theorems.
\newblock {\em Statistics and Probability Letters}, 135:7--10, 2018.

\bibitem{DefAssoc}
B.~B{\l}aszczyszyn and D.~Yogeshwaran.
\newblock Clustering comparison of point processes with applications to random
  geometric models.
\newblock In V.~Schmidt, editor, {\em Stochastic Geometry, Spatial Statistics
  and Random Fields: Models and Algorithms}, volume 2120 of {\em Lecture Notes
  in Mathematics}, chapter~2, pages 31--71. Springer, 2014.

\bibitem{Bolt}
E.~Bolthausen.
\newblock On the central limit theorem for stationnary mixing random fields.
\newblock {\em The Annals of Probability}, 10:1047--1050, 1982.

\bibitem{Survey}
R.C. Bradley.
\newblock Basic properties of strong mixing conditions. a survey and some open
  questions.
\newblock {\em Probability Surveys}, 2:107--144, 2005.

\bibitem{Bulinski}
A.~Bulinski and E.~Shabanovich.
\newblock Asymptotic behaviour of some functionals of positively and negatively
  dependent random fields.
\newblock {\em Fundam. Prikl. Mat.}, 4:479--492, 1998.

\bibitem{BulinskiRate}
A.~Bulinski and A.~Shashkin.
\newblock {\em Limit Theorems for Associated Random Fields and Related
  Systems}.
\newblock Springer-Verlag, 2007.

\bibitem{AppNa1}
R.~Burton and E.~Waymire.
\newblock Scaling limits for associated random measures.
\newblock {\em Ann. Appl. Probab.}, 13:1267--1278, 1985.

\bibitem{DV}
D.J. Daley and D.~Vere-Jones.
\newblock {\em An Introduction to the Theory of Point Processes, Volume I:
  Elementary Theory and Methods}.
\newblock Springer, New York, second edition, 2003.

\bibitem{Davydov}
Yu.~A. Davydov.
\newblock The convergence of distributions which are generated by stationary
  random processes.
\newblock {\em Teor. Verojatnost. i Primenen.}, 13:730--737, 1968.

\bibitem{mono}
C.~Dellacherie and P.~Meyer.
\newblock {\em Probabilities and Potential}.
\newblock Elsevier Science Ltd, 1978.

\bibitem{Telecom1}
N.~Deng, W.~Zhou, and M.~Haenggi.
\newblock The ginibre point process as a model for wireless networks with
  repulsion.
\newblock {\em IEEE Transactions on Wireless Communications}, 1:479--492, 2015.

\bibitem{Doukhan}
P.~Doukhan.
\newblock {\em Mixing: Properties and Examples}.
\newblock Springer-Verlag, 1994.

\bibitem{DiscreteDep}
P.~Doukhan, K.~Fokianos, and X.~Li.
\newblock On weak dependence conditions: the case of discrete valued processes.
\newblock {\em Statist. Probab. Lett.}, 82(11):1941--1948, 2012.

\bibitem{WeakDep}
P.~Doukhan and S.~Louhichi.
\newblock A new weak dependence condition and applications to moment
  inequalities.
\newblock {\em Stochastic Process. Appl.}, 84:313--342, 1999.

\bibitem{PA_Int}
J.D. Esary, F.~Proschan, and D.W. Walkup.
\newblock Association of random variables, with applications.
\newblock {\em Ann. Math. Statist.}, 38:1466--1474, 1967.

\bibitem{AppNa2}
S.N. Evans.
\newblock Association and random measures.
\newblock {\em Probability Theory and Related Fields}, 86:1--19, 1990.

\bibitem{NA}
S.~Ghosh.
\newblock {Determinantal processes and completeness of random exponentials: the
  critical case}.
\newblock {\em Probability Theory and Related Fields}, 163(3):643--665, 2015.

\bibitem{gomez_case_2015}
J.~S. Gomez, A.~Vasseur, A.~Vergne, P.~Martins, L.~Decreusefond, and W.~Chen.
\newblock A {Case} {Study} on {Regularity} in {Cellular} {Network}
  {Deployment}.
\newblock {\em IEEE Wireless Communications Letters}, 4(4):421--424, 2015.

\bibitem{Guyon}
X.~Guyon.
\newblock {\em Random Fields on a Network}.
\newblock Springer-Verlag, 1994.

\bibitem{Zweirich}
L.~Heinrich.
\newblock On the strong {B}rillinger-mixing property of
  {$\alpha$}-determinantal point processes and some applications.
\newblock {\em Appl. Math.}, 61(4):443--461, 2016.

\bibitem{StellaKleinempirical2011}
L.~Heinrich and S.~Klein.
\newblock Central limit theorems for empirical product densities of stationary
  point processes.
\newblock {\em Statistical Inference for Stochastic Processes. An International
  Journal Devoted to Time Series Analysis and the Statistics of Continuous Time
  Processes and Dynamical Systems}, 17(2):121--138, 2014.

\bibitem{Hough}
J.B. Hough, M.~Krishnapur, Y.~Peres, and B.~Virag.
\newblock {\em Zeros of Gaussian Analytic Functions and Determinantal Point
  Processes}.
\newblock American Mathematical Society, 2009.

\bibitem{Ibragimov}
I.A. Ibragimov and Y.V. Linnik.
\newblock {\em Independant and stationnary sequences of random variables}.
\newblock Wolters-Noordhoff, 1971.

\bibitem{jolivetTCL}
E.~Jolivet.
\newblock Central limit theorem and convergence of empirical processes for
  stationary point processes.
\newblock In {\em Point processes and queuing problems (Colloqium, {K}eszthely,
  1978)}, volume~24 of {\em Colloq. Math. Soc. J\'anos Bolyai}, pages 117--161.
  North-Holland, Amsterdam-New York, 1981.

\bibitem{MachLearn}
A.~Kulesza and B.~Taskar.
\newblock Determinantal point process models for machine learning.
\newblock {\em Foundations and Trends in Machine Learning}, 5:123--286, 2012.

\bibitem{Lavancier}
F.~{Lavancier}, J.~{M{\o}ller}, and E.~{Rubak}.
\newblock Determinantal point process models and statistical inference.
\newblock {\em Journal of Royal Statistical Society: Series B (Statistical
  Methodology)}, 77:853--877, 2015.

\bibitem{Lyons}
R.~Lyons.
\newblock Determinantal probability: Basic properties and conjectures.
\newblock {\em Proceedings of the International Congress of Mathematicians,
  Seoul, Korea}, IV:137--161, 2014.

\bibitem{Macchi}
O.~Macchi.
\newblock The coincidence approach to stochastic point processes.
\newblock {\em Advances in Applied Probability}, 7:83--122, 1975.

\bibitem{Telecom2}
N.~Miyoshi and T.~Shirai.
\newblock A cellular network model with ginibre configured base stations.
\newblock {\em Advances in Applied Probability}, 46:832--845, 2014.

\bibitem{moeller:waagepetersen:04}
J.~M{\o}ller and R.P. Waagepetersen.
\newblock {\em Statistical Inference and Simulation for Spatial Point
  Processes}.
\newblock Chapman and Hall/CRC, Boca Raton, 2004.

\bibitem{Rio}
E.~Rio.
\newblock Covariance inequalities for strongly mixing processes.
\newblock {\em Ann. Inst. H. Poincar\'e Probab. Statist.}, 29(4):587--597,
  1993.

\bibitem{AlphaMix}
M.~Rosenblatt.
\newblock A central limit theorem and a strong mixing condition.
\newblock {\em Probc. Nat. Acad. Sci. U.S.A.}, 42:43--47, 1956.

\bibitem{Shirai}
T.~Shirai and Y.~Takahashi.
\newblock Random point fields associated with certain fredholm determinants i:
  fermion, poisson and boson point processes.
\newblock {\em Journal of Functional Analysis}, 205:414--463, 2003.

\bibitem{Sosh}
A.~Soshnikov.
\newblock Determinantal random point fields.
\newblock {\em Russian Math. Surveys}, 55:923--975, 2000.

\bibitem{waag}
R.~Waagepetersen and Y.~Guan.
\newblock Two-step estimation for inhomogeneous spatial point processes.
\newblock {\em Journal of the Royal Statistical Society}, 71:685--702, 2009.

\bibitem{CLTRFNA}
M.~Yuan, C.~Su, and T.~Hu.
\newblock A central limit theorem for random fields of negatively associated
  processes.
\newblock {\em Journal of Theoretical Probability}, 16:309--323, 2003.

\end{thebibliography}

\end{document}